\documentclass[11pt,reqno,tbtags]{amsart}
\usepackage{amssymb}
\usepackage{ifdraft}
\usepackage[square,numbers]{natbib}
\numberwithin{equation}{section}

\allowdisplaybreaks

\newtheorem{theorem}{Theorem}[section]
\newtheorem{lemma}[theorem]{Lemma}
\newtheorem{proposition}[theorem]{Proposition}
\newtheorem{corollary}[theorem]{Corollary}

\theoremstyle{definition}
\newtheorem{example}[theorem]{Example}
\newtheorem*{definition}{Definition}
\newtheorem{problem}[theorem]{Problem}
\newtheorem{remark}[theorem]{Remark}

\theoremstyle{remark}

\newenvironment{romenumerate}[1][0pt]{% optional argument changes indentation
\addtolength{\leftmargini}{#1}\begin{enumerate}% gives (i), (ii) etc.
 }{\end{enumerate}}

\newcounter{oldenumi}
{\setcounter{oldenumi}{\value{enumi}}
\begin{romenumerate} \setcounter{enumi}{\value{oldenumi}}}
{\end{romenumerate}}

\newcounter{thmenumerate}
\newenvironment{thmenumerate}
{\setcounter{thmenumerate}{0}%
 \def\item{\par% \ifnum\thethmenumerate=0\else\par\fi %\noindent\fi
 \refstepcounter{thmenumerate}\textup{(\roman{thmenumerate})\enspace}}
}
{}

\newcounter{romxenumerate}   %less indented than standard.

\newcounter{xenumerate}   %no left indentation; thus wider lines

\newcommand\pfitem[1]{\par(#1):}

\newcommand{\refT}[1]{Theorem~\ref{#1}}
\newcommand{\refC}[1]{Corollary~\ref{#1}}
\newcommand{\refL}[1]{Lemma~\ref{#1}}
\newcommand{\refR}[1]{Remark~\ref{#1}}
\newcommand{\refS}[1]{Section~\ref{#1}}

\newcommand{\refP}[1]{Proposition~\ref{#1}}

\newcommand{\refE}[1]{Example~\ref{#1}}

\newcommand{\refand}[2]{\ref{#1} and~\ref{#2}}

\begingroup
  \divide\count255 by 60
  \multiply\count255 by -60
  \advance\count255 by \time
  \ifnum \count255 < 10 \xdef\klockan{\the\count1.0\the\count255}
  \else\xdef\klockan{\the\count1.\the\count255}\fi
\endgroup

\DeclareMathOperator*{\setdiff}{\triangle}

\newcommand{\sumin}{\sum_{i=1}^n}
\newcommand{\sumjn}{\sum_{j=1}^n}
\newcommand\set[1]{\ensuremath{\{#1\}}}
\newcommand\bigset[1]{\ensuremath{\bigl\{#1\bigr\}}}
\newcommand\Bigset[1]{\ensuremath{\Bigl\{#1\Bigr\}}}

\newcommand\bigpar[1]{\bigl(#1\bigr)}
\newcommand\Bigpar[1]{\Bigl(#1\Bigr)}

\newcommand\lrpar[1]{\left(#1\right)}

\newcommand\abs[1]{|#1|}
\newcommand\bigabs[1]{\bigl|#1\bigr|}
\newcommand\Bigabs[1]{\Bigl|#1\Bigr|}
\newcommand\biggabs[1]{\biggl|#1\biggr|}
\newcommand\lrabs[1]{\left|#1\right|}
\def\rompar(#1){\textup(#1\textup)}    % usage: \rompar(...)
\newcommand\xfrac[2]{#1/#2}

\def\xexp(#1){e^{#1}}
\newcommand\ceil[1]{\lceil#1\rceil}
\newcommand\bigceil[1]{\bigl\lceil#1\bigr\rceil}

\newcommand\nutoo{\ensuremath{{\nu\to\infty}}}
\newcommand\ntoo{\ensuremath{{n\to\infty}}}

\newcommand\itoo{\ensuremath{{i\to\infty}}}

\newcommand\norm[1]{\|#1\|}

\newcommand\upto{\nearrow}
\newcommand\punkt{.\spacefactor=1000}    % om problem!
\newcommand\iid{i.i.d\punkt}    
\newcommand\ie{i.e\punkt}
\newcommand\eg{e.g\punkt}

\newcommand\cf{cf\punkt}
\newcommand{\as}{a.s\punkt}
\newcommand{\aex}{a.e\punkt}
\newcommand\bbR{\mathbb R}

\newcounter{CC}
\newcounter{cc}
\newcommand\E{\operatorname{\mathbb E{}}}
\renewcommand\P{\operatorname{\mathbb P{}}}

\newcommand\ga{\alpha}

\newcommand\gf{\varphi}

\newcommand\gG{\Gamma}

\newcommand\gO{\Omega}
\newcommand\gs{\sigma}

\newcommand\eps{\varepsilon}

\newcommand\cA{\mathcal A}

\newcommand\cC{\mathcal C}

\newcommand\cF{\mathcal F}

\newcommand\cP{\mathcal P}

\newcommand\cS{{\mathcal S}}
\newcommand\cT{{\mathcal T}}

\newcommand\cW{\mathcal W}

\newcommand\ett[1]{\boldsymbol1\set{#1}}

\newcommand\etta{\boldsymbol1} 
\def\[#1]{[\![#1]\!]}

\newcommand\qw{^{-1}}
\newcommand\qww{^{-2}}
\newcommand\qwww{^{-3}}
\newcommand\qq{^{1/2}}

\renewcommand{\=}{:=}

\newcommand\intoi{\int_0^1}

\newcommand\intoooo{\int_{-\infty}^\infty}
\newcommand\oi{[0,1]}
\newcommand\oii{[0,1]^2}
\newcommand\setoi{\set{0,1}}

\newcommand\dd{\,\textup{d}}

\newcommand\cuoo{\mathcal U_{\infty}}
\newcommand\cum{\mathcal U_{\uparrow}}
\newcommand\cwm{\mathcal W_{\uparrow}}
\newcommand\qm{quasimonotone}
\newcommand\gn{G_\nu}
\newcommand\gnn{(\gn)}
\newcommand\gnx{G_n}
\newcommand\gnxx{(\gnx)}
\newcommand\normll[1]{\norm{#1}_{L^1}}
\newcommand\normlss[1]{\norm{#1}\qliss}

\newcommand\normoo[1]{\norm{#1}_{\infty}}

\newcommand\cn[1]{\norm{#1}\cut}

\newcommand\cnone[1]{\norm{#1}_{\square,1}}
\newcommand\cutnorm{\cn}

\newcommand\cut{_{\square}}
\newcommand\dcut{{\delta_{\square}}}
\newcommand\dcutone{{\delta_{\square,1}}}

\newcommand\dl{{\delta_{1}}}

\newcommand\sss{{\mathcal S}}
\newcommand\sssq{{\mathcal S^2}}
\newcommand\hsss{\widehat{\mathcal S}}
\newcommand\hprec{\widehat{\prec}}

\newcommand\liss{{L^1(\sss^2)}}

\newcommand\lioi{{L^1(\oi^2)}}

\newcommand\qliss{_{L^1(\sss^2)}}
\newcommand\qlis{_{L^1(\sss)}}
\newcommand\qlioi{_{L^1(\oi^2)}}

\newcommand\www{\cW}
\newcommand\bW{\overline W}
\newcommand\hW{\widehat{W}}
\newcommand\ints{\int_\sss}

\newcommand\mpp{measure-preserving}

\newcommand\sign{\operatorname{sign}}
\newcommand\wa{W_A}

\newcommand\wx[1]{W_{#1}}
\newcommand\mW{W_\sss}%\ws
\newcommand\ws{\wx{\sss}}
\newcommand\wis{\wx{1,\sss}}
\newcommand\wiis{\wx{2,\sss}}
\newcommand\wisi{\wis}
\newcommand\ops{ordered probability space}
\newcommand\ps{probability space}
\newcommand\pss{probability spaces}
\newcommand\wn[1]{w^{(n)}_{#1}}
\newcommand\sssqq{\sss_1\times\sss_2}
\newcommand\sssp{(\sss,\prec)}
\newcommand\sssmp{(\sss,\mu,\prec)}
\newcommand\precx{\prec^*}
\newcommand\gox{\gO^*}
\newcommand\goxa{\gO^*_0}
\newcommand\goxb{\gO^*_1}
\newcommand\goxx{\widetilde\gO^*}
\newcommand\oivalued{$0/1$-valued}  %??
\newcommand\de{d_{\mathrm{e}}}% was textsf
\newcommand\wgn{W_{\gn}}
\newcommand\wgni{W_{\gn'}}
\newcommand\muu{\mu\times\mu}
\newcommand\noproof{\qed} %\square$}
\newcommand\be{{\overline e}}
\newcommand\ccc{\mathrm{c}}
\newcommand\gOa{\gO_0}% was \gO_1
\newcommand\gOb{\gO_1}% was \gO_3
\newcommand\gOc{\gO_2}% was \gO_3'
\newcommand\gOd{\gO_1'}% was \gO_4
\newcommand\gOkb{\gO_1}% for kernels
\newcommand\gOkc{\gO_2}% for kernels
\newcommand\wv{W^{V}}
\newcommand{\Lovasz}{Lov\'asz}

\newcommand\REM[1]{{\raggedright\texttt{[#1]}\par\marginal{XXX}}}

\newenvironment{comment}{\setbox0=\vbox\bgroup}{\egroup} %deletes!

\newcommand\urladdrx[1]{{\urladdr{\def~{{\tiny$\sim$}}#1}}}

\begin{document}
\title%[]
{Monotone graph limits and quasimonotone graphs}

\date{21 January, 2011}  %; revised ...

\author{B\'ela Bollob\'as}
\address{Department of Pure Mathematics and Mathematical Statistics,
University of Cambridge,
Wilberforce Road, Cambridge CB3 0WB, UK 
and
Department of Mathematical Sciences,
University of Memphis, Memphis TN 38152, USA}
\thanks{The first author's research was supported in part by 
NSF grants  CNS-0721983, CCF-0728928 and DMS-0906634, 
and ARO grant W911NF-06-1-0076.}
\email{b.bollobas@dpmms.cam.ac.uk}

\author{Svante Janson}
\address{Department of Mathematics, Uppsala University, PO Box 480,
SE-751~06 Uppsala, Sweden}
\email{svante.janson@math.uu.se}
\urladdrx{http://www.math.uu.se/~svante/}
\thanks{Part of this research was carried when
SJ visited the Isaac Newton Institute, Cambridge,
during the programme Stochastic Processes in Communication Sciences, 2010}

\author{Oliver Riordan}
\address{Mathematical Institute, University of Oxford, 24--29 St Giles',
  Oxford OX1 3LB, UK} 
\email{riordan@maths.ox.ac.uk}

%\keywords{<keywords>}
%{60C05 (68P10,68W40)} %%{Primary: <subject>; Secondary: <subject>}

\begin{comment}  % Some suggestions:
05 Combinatorics 
05C Graph theory [For applications of graphs, see 68R10, 90C35, 94C15]
05C05 Trees
05C07 Vertex degrees
05C35 Extremal problems [See also 90C35]
05C40 Connectivity
05C65 Hypergraphs
05C80 Random graphs
05C90 Applications
05C99 None of the above, but in this section 

\end{comment}

\begin{abstract} 
  The recent theory of graph limits gives a powerful framework for
  understanding the properties of suitable (convergent) sequences
  $(G_n)$ of graphs in terms of a limiting object which may be
  represented by a symmetric function $W$ on $[0,1]$, i.e., a {\em kernel} or {\em graphon}.
  In this
  context it is natural to wish to relate specific properties of the
  sequence to specific properties of the kernel. Here we show that the
  kernel is monotone (i.e., increasing in both variables) if and only
  if the sequence satisfies a `quasi-monotonicity' property defined by
  a certain functional tending to zero. As a tool we prove an
  inequality relating the cut and $L^1$ norms of kernels of the form
  $W_1-W_2$ with $W_1$ and $W_2$ monotone that may be of
  interest in its own right; no such inequality holds for general kernels.
\end{abstract}

\maketitle

\section{Introduction}\label{S:intro}

Recently, \citet{LSz} and Borgs, Chayes, Lov\'asz, S\'os and Vesztergombi
(see, e.g.,~\cite{BCLSV1}) developed a rich theory of \emph{graph limits},
associating limit objects to suitable sequences $(\gn)$ of (dense)
graphs with $|\gn|\to\infty$, where $|\gn|$ denotes the number of vertices
of $\gn$. The basics of this theory are outlined in \refS{Slimits} below;
see also \citet{SJ209}.
These graph limits (which are not themselves graphs) can be represented in
  several different ways; perhaps the most important is that every graph
  limit can be represented by a \emph{kernel} (or \emph{graphon}) on \oi,
  \ie, a symmetric measurable function $W:\oi^2\to\oi$. However,
this representation is in general not unique, see
\eg{} \cite{LSz,BCL:unique,SJ209,BR:metrics}.
More generally, kernels can be defined on any probability space, see
  \refS{Slimits}.

We use $\gG$ to denote an arbitrary graph limit, and write $\gG_W$ for the graph
  limit defined by a kernel $W$. We say that two kernels $W$ and $W'$ 
are \emph{equivalent} if they define the same graph limit, \ie, if
$\gG_W=\gG_{W'}$. We write $\gn\to\gG$ when the sequence $\gnn$ converges to
  $\gG$ (see \cite{LSz},  \cite{BCLSV1} and \refS{Slimits} below for
  definitions);  
if $\gG$ is represented by a kernel $W$, \ie, if $\gG=\gG_W$,
  we also  write $\gn\to W$.

Following \cite{SJ209}, we denote the set of all graph limits by $\cuoo$,
and note that $\cuoo$ is a compact metric space. 
Another version of the important compactness property for graph limits
is that every sequence $\gnn$ of graphs with 
$|\gn|\to\infty$ has a convergent subsequence, \ie, a subsequence converging
to some $\gG\in\cuoo$. 

Given a suitable class $\cF$ of graphs, it seems interesting to study
the {\em graph limits of $\cF$}, i.e.,
the set of graph limits arising as limits of sequences of graphs in $\cF$.
One interesting example is
the class of \emph{threshold graphs}, which has several different
characterizations, see \eg\ \cite{MP}. One of them is the monotonicity
property of the neighbourhoods $N(v)$ of the vertices: 
\begin{equation}\label{tg}
\vbox{\hsize=10cm\noindent 
There exists a
(linear) ordering $\prec$ of the vertices such that \\
if $v\prec w$, then 
$N(v)\setminus\set{v,w}\subseteq N(w)\setminus\set{v,w}$.}
\raisetag{2\baselineskip}
\end{equation}

The graph limits of threshold graphs were studied by \citet{SJ238} (see also
\cite{LSz:finitely}), who showed that they
are exactly the graph limits that can be represented by kernels $W$ that
take values in \setoi{} only and are \emph{increasing}, in that
\begin{equation}\label{q2}
W(x_1,y_1)\le W(x_2,y_2)
\qquad \text{if }
0\le x_1\le x_2\le 1,\,
0\le y_1\le y_2\le 1.
\end{equation}
In other words, $W$ is the indicator function of a symmetric increasing subset of
$\oii$. 
(In this paper, `increasing' should always be interpreted in the weak sense,
\ie, as `non-decreasing'.) 
Moreover, the representation by such a $W$ is unique, if, as is usual, we
identify functions that are equal a.e. 

Note that the monotonicity properties in \eqref{tg}
and \eqref{q2} are obviously related; this is perhaps best seen if
\eqref{tg} is rewritten as a monotonicity property of the adjacency matrix
of the graph (with some exceptions at the diagonal), so even without the
detailed technical study in \cite{SJ238}, the condition \eqref{q2} should
not be surprising.

Increasing and decreasing kernels define the same set of graph limits, by the change of
variables $x\mapsto1-x$. Hence we shall talk about \emph{monotone} kernels
rather than increasing kernels, but for simplicity (and without loss of generality)
we consider only increasing ones, so in this paper `monotone' is regarded
as synonymous with `increasing'. 

The main purpose of the present paper is to study the larger class of graph
limits represented by arbitrary monotone kernels (taking any values in
\oi, rather than just the values 0 and 1), and the corresponding sequences of graphs.
We shall also study analytic properties of monotone kernels themselves.
\begin{definition}
  Let $\cwm$ be the set of monotone kernels on $\oi$, \ie,
the set of all symmetric measurable functions $W:\oii\to\oi$ that satisfy \eqref{q2}.

Let $\cum$ be the
corresponding class of graph limits, \ie, the class of graph limits that can
be represented as $\gG_W$ for some $W\in\cwm$. We call these graph limits
\emph{monotone}. 
\end{definition}

By definition, every monotone graph limit can be represented by a monotone kernel $W$
on $\oi$, but note that a monotone graph limit 
may also have many representations by
non-monotone kernels. For example, a monotone kernel can be rearranged by
an arbitrary \mpp{} bijection from $\oi$ to itself, which will
in general destroy monotonicity.

The classes $\cwm$ of monotone kernels and $\cum$ of monotone graph limits
are studied in \refS{Smono}. We show there that $\cwm$ is a compact subset
of $L^1(\oii)$, and that $\cum$ is a compact subset of $\cuoo$. In addition, we
consider monotone kernels defined on other (ordered) probability
spaces, showing that each such kernel is equivalent to a monotone kernel on
$\oi$, so the class $\cum$ is not enlarged by allowing arbitrary \ps{s}.

\begin{definition}
A sequence $(\gn)$ of graphs with $|\gn|\to\infty$ is
\emph{quasimonotone} if it converges to the set $\cum$, in the sense that each
convergent subsequence has as its limit a graph limit in $\cum$. 
In this case we will also say that $(\gn)$ is a sequence of quasimonotone graphs.
\end{definition}
In particular, a sequence $(\gn)$ converging
to a graph limit in $\cum$ is quasimonotone.
Note that it makes no formal sense to ask whether an individual graph
is quasimonotone; just as for quasirandomness, quasimonotonicity
is a property of sequences of graphs.

\begin{example}[Threshold graphs are quasimonotone]
  As noted above, each convergent sequence of threshold graphs converges to
  a limit represented by a \oivalued{} 
kernel $W\in\cwm$. Hence every
  sequence of threshold graphs (with orders tending to $\infty$) is \qm.
\end{example}

\begin{example}[Quasirandom graphs are quasimonotone]
  Quasi\-random graphs were introduced by \citet{Thomason87a,Thomason87b}
as sequences $(\gn)$ of graphs that have certain properties typical of random
graphs. A number of different such properties turn out to be equivalent, and
there are thus many equivalent characterizations, see \citet{ChungGW:quasi}.
Another characterization, found by \citet{LSz},
is that a sequence $(\gn)$ is quasirandom if and only if it converges to a graph limit 
represented by a constant kernel $W(x,y)=p$, for some $p\in\oi$. 
(See also \cite{LSos} and \cite{SJ234}.)
Since a
constant function is monotone, $W\in\cwm$, and thus every 
quasirandom sequence of graphs is \qm.
\end{example}

\begin{example}[Random graphs are quasimonotone]
  The sequence of random graphs  $G(\nu,p)$ 
with some fixed $p\in\oi$ and $\nu=1,2,\dots$ (coupled in
  the natural way for different $\nu$) is \as{}
  quasirandom, and thus \as{} \qm.
\end{example}

Our main result (\refT{TQ} below)
is that \qm{} graphs can be characterized by a weakening of
\eqref{tg}. As is typical for conditions concerning convergence to graph limits,
this weakening involves taking
averages over subsets of the vertex set $V$, rather than imposing
a condition for all vertices, and allows for a small `error',
making the condition asymptotic.

Given a graph $G$ with vertex set $V=V(G)$, a vertex $v$ of $G$ and a subset $A$ of $V$, let
\begin{equation*}
 e(v,A)\=| N(v)\cap A|=|\set{w\in A:w\sim v}|
\end{equation*}
denote the number of edges from $v$ to $A$.

Let $x_+$ denote the {\em positive part} of $x$, i.e., $\max\{x,0\}$.
Writing $n\=|G|=|V|$, given
a (linear) order $\prec$ on $V$ and a subset $A\subseteq V$, define
\begin{align}
  \gOa(G,\prec,A)
&\=
\frac1{n^3}\sum_{v\prec w}\bigpar{e(v,A\setminus\set{w})
 -e(w,A\setminus\set{v})}_+
\label{go1<a1}
\\
&\phantom:=
\frac1{n^3}\sum_{v\prec w}\bigpar{e(v,A\setminus\set{v,w})
  -e(w,A\setminus\set{v,w})}_+
,\label{go1<a2}
\\
  \gOa(G,\prec) &\=\max_{A\subseteq V}   \gOa(G,\prec,A), \text{\ \ and}
\label{go1<}\\
  \gOa(G) &\=\min_{\prec}   \gOa(G,\prec).
\label{go1}
\end{align}
In the last line the minimum is taken over all $n!$ orders on $V$.
The normalization by $n^3$ ensures that $0\le\gOa<1$. In fact, $\gOa<1/2$, and
this bound can be improved further, but this is not important for our purposes
since we are interested in small values of $\gOa$.

Note that $\gOa(G)=0$ if and only if there exists an order $\prec$ such
that $\gOa(G,\prec,A)=0$ for every $A$, \ie, 
$e(v,A\setminus\set{v,w})\le e(w,A\setminus\set{v,w})$ for all $A$ and $v\prec
w$, which easily is seen to be equivalent to \eqref{tg}, giving the following
result.
\begin{proposition}\label{Pgo10}
A graph $G$ is a threshold graph if and only if\/ 
  $\gOa(G)=0$.\noproof
\end{proposition}
Note that $\gOa$ is not intended as a measure of how far a
graph
is from being a threshold graph (for such a measure, see \refS{Sqt}).
Rather, we may think (informally!) of a typical quasimonotone graph
as being similar to a random graph in which edges are independent, and the probability
$p_{ij}$ of an edge $ij$ is increasing in $i$ and in $j$. In such a graph,
one cannot expect the neighbourhoods of different vertices to be even approximately
nested. But one can expect that for all `large' sets $A$ of vertices,
for most $i<j$, $e(i,A)$ will be smaller than (or at least not much larger than)
$e(j,A)$. The idea is that a small value of $\gOa(G)$ detects
this phenomenon, without relying on any given labelling of the vertices.

Some variations of the functional $\gOa$ will be defined
in \refS{SB}, where we shall show that they are asymptotically equivalent
for our purposes.

Our main result is the following, proved 
in \refS{SpfTQ}.
(All unspecified limits in this paper are taken as \nutoo.)
\begin{theorem}
  \label{TQ}
Let $(\gn)$ be a sequence of graphs with $|\gn|\to\infty$. Then $(\gn)$ is
\qm{} if and only if\/ $\gOa(\gn)\to0$.
\end{theorem}

We state a special case separately.

\begin{theorem}
  \label{TQ1}
Let $(\gn)$ be a sequence of graphs with $|\gn|\to\infty$, and suppose that 
$(\gn)$ is convergent, i.e., $\gn\to \gG$ for some graph limit $\gG\in\cuoo$.
Then $\gG\in\cum$ if and only if\/ $\gOa(\gn)\to0$.
\end{theorem}

We give several results on monotone graph limits in 
Sections \ref{Smono}--\ref{Spfmono2}. These include a characterization 
in terms of a functional $\gO(W)$ for kernels, analoguous to $\gOa$ for
graphs. Along the way we prove some results about monotone kernels
that may be of interest in their own right. For example,
on functions that may be written as the difference between
two monotone kernels, the $L^1$ norm and the cut norm may be
bounded in terms of each other; see \refT{TV3}.

\begin{remark}\label{RLSz}
  \citet{LSz:topology} have studied the class of graph limits represented by 
\oivalued{} kernels (and the corresponding graph properties); with a slight
  variation of their terminology
we call such graph limits \emph{random-free}. 
In contrast to the monotone case, it can be shown that {\em every}
representing kernel of a random-free limit is \aex{}
\oivalued; see \cite{SJ249}.
It follows that the graph
limits that are both monotone and random-free are exactly the threshold
graph limits.
\end{remark}

In \refS{Sqt}, we consider the functional obtained by taking the supremum
over $A$ inside the sum in \eqref{go1<a1} instead of outside as in
\eqref{go1<}. We shall show that this stronger functional characterizes
convergence to threshold graph limits instead of monotone graph limits;
we call the corresponding sequences of graphs \emph{quasithreshold}.

\subsection{A problem}

The convergence $\gn\to\gG$ of a sequence $(\gn)$ of graphs to a graph limit
$\gG$ can be expressed using the \emph{homomorphism numbers} $t(F,\cdot)$:
$\gn\to\gG$ if and only if
$t(F,\gn)\to t(F,\gG)$ for every fixed graph $F$; see \eg{} 
\cite{LSz},
\cite{BCLSV1},
\cite{SJ209}
for definitions and further results.
In particular, the graph limit $\gG$ is characterized by the family
$(t(F,\gG))_F$.
The
families $(t(F,\gG))_F$ that appear are characterized algebraically by
\citet{LSz}. 

\begin{problem}
  Characterize the families $(t(F,\gG))_F$ that appear for $\gG\in\cum$.
\end{problem}

\medskip

The rest of this paper is organized as follows. In the next section we
review some basic properties of the cut metric that we shall
rely on throughout the paper. In \refS{SB} we introduce some variants
of the functional $\gOa$ for graphs. In \refS{Smono} we define
analogous functionals for kernels and state several key properties;
these are proved in the next two sections, and
then our main results are deduced in \refS{SpfTQ}. Finally,
in \refS{Sqt} we discuss related functionals characterizing quasithreshold graphs.

\section{Kernels and graph limits}\label{Slimits}

We state here some standard definitions and results that we shall use
later in the paper. For proofs and further details, see \eg{} \citet{BCLSV1},
\citet{BR:metrics}, or \citet{SJ210,SJ249}.

Let $(\sss,\cF,\mu)$ be a probability space;
for simplicity, we will usually abbreviate the notation to $\sss$ or
$(\sss,\mu)$.

A \emph{kernel} (or \emph{graphon}) on $\sss$ is a symmetric measurable
function $\sssq\to\oi$. We let $\www(\sss)$ denote the set of all kernels on
$\sss$. 

If $W$ is an integrable function on $\sssq$, we define its 
\emph{cut norm} 
by 
\begin{equation}\label{cutnorm}
 \cn W \= \sup_{\normoo{f},\normoo{g}\le1}
  \Bigabs{\int_{\sss^2} W(x,y)f(x)g(y)\dd\mu(x)\dd\mu(y)},
\end{equation}
where $\normoo \cdot$ denotes the norm in $L^\infty$.
In other words, the supremum in \eqref{cutnorm}
is taken over all (real-valued) functions $f$ and $g$ with
values in $[-1,1]$.
(Several other versions exist, which are equivalent within constants.)
By considering the supremum over $f$ with $g$ fixed,
and vice versa, it is easy to see that the supremum is unchanged
if we restrict $f$ and $g$ to take values in $\{\pm 1\}$, so we have
\begin{equation}\label{cutnormpm}
 \cn W = \sup_{f,g:\sss\to\{\pm 1\}}
  \Bigabs{\int_{\sss^2} W(x,y)f(x)g(y)\dd\mu(x)\dd\mu(y)}.
\end{equation}
This norm defines a metric $\cn{W_1-W_2}$ for kernels on the same probability
space $\sss$; as usual, we identify kernels that are equal a.e.

The cut norm may be used to define another (semi)metric $\dcut$, the
\emph{cut metric},
as  follows.
If $\gf:\sss_1\to\sss_2$ is a \mpp{} map between two probability
spaces and $W$ is a kernel on $\sss_2$, we let $W^\gf$ be the kernel on
$\sss_1$ defined by $W^\gf(x,y)\=W\bigpar{\gf(x),\gf(y)}$.
Let $W_1$ be a kernel on a \ps{} $\sss_1$ and $W_2$ a kernel on a possibly
different \ps{} $\sss_2$. Then
\begin{equation}\label{dcut}
\dcut(W_1,W_2)
\=\inf_{\gf_1,\gf_2}\cutnorm{W_1^{\gf_1}-W_2^{\gf_2}},
\end{equation}
where the infimum is taken over all \emph{couplings} $(\gf_1,\gf_2)$ of $\sss_1$ and
$\sss_2$, \ie, over all pairs of \mpp{} maps $\gf_1:\sss_3\to\sss_1$ and
$\gf_2:\sss_3\to\sss_2$ from a third \ps{} $\sss_3$.
It is not difficult to verify that $\dcut$ satisfies the triangle
inequality (see \eg{} \cite{SJ249}), but note that $\dcut(W_1,W_2)$ may be 0
even if $W_1\neq W_2$, 
for example if $W_1=W_2^\gf$ for some \mpp{}
$\gf:\sss_1\to\sss_2$. Hence, $\dcut$ is really a semimetric (but is
usually called a metric for simplicity).

Note that $\dcut(W_1,W_2)$ is defined for kernels on different
spaces. Moreover, it is invariant under \mpp{} maps: 
$\dcut(W_1^{\gf_1},W_2^{\gf_2})=\dcut(W_1,W_2)$ for any \mpp{} maps
$\gf_k:\sss'_k\to\sss_k$, $k=1,2$. 

Although we allow couplings $(\gf_1,\gf_2)$ defined on an arbitrary third
space $\sss_3$, in \eqref{dcut} it suffices to consider the case when
$\sss_3=\sss_1\times\sss_2$, with a measure $\mu$ having marginals $\mu_1$
and $\mu_2$, taking for $\gf_1$ and $\gf_2$ the
projections $\pi_k:\sss_1\times\sss_2\to\sss_k$, $k=1,2$. In fact, for an
arbitrary coupling $(\gf_1,\gf_2)$ defined on a space $(\sss_3,\mu_3)$, the
mapping
$(\gf_1,\gf_2):\sss_3\to\sss_1\times\sss_2$ maps $\mu_3$ to a measure $\mu$
on $\sss_1\times\sss_2$ with the right marginals, and it is easily seen that 
$\cn{W_1^{\gf_1}-W_2^{\gf_2}}=\cn{W_1^{\pi_1}-W_2^{\pi_2}}$.

Although this will be of much lesser importance, we also define
the corresponding
rearrangement-invariant version of the $L^1$ distance:
\begin{equation}\label{dl}
\dl(W_1,W_2)\=\inf_{\gf_1,\gf_2}\norm{W_1^{\gf_1}-W_2^{\gf_2}}_{L^1(\sss_3^2)}.
\end{equation}

The coupling definition \eqref{dcut} of the cut metric
is valid for all $\sss_1$ and $\sss_2$,
but in common special cases it is possible, and often convenient, to use other,
equivalent, definitions.
For example, if $\sss_1=\sss_2=\oi$ (equipped with the Lebesgue
measure, as always), then
as shown by \citet[Lemma 3.5]{BCLSV1},
\begin{equation}\label{dcut2}
\dcut(W_1,W_2)
\=\inf_{\gf}\cutnorm{W_1-W_2^{\gf}},
\end{equation}
taking the infimum  over all \mpp{} bijections $\oi\to\oi$.

We say that two kernels $W_1$ and $W_2$ are \emph{equivalent} if
$\dcut(W_1,W_2)=0$. 
The set of equivalence classes is thus a metric space with the metric
  $\dcut$. 
A central result \cite{LSz,BCLSV1} is that these equivalence classes are in
one-to-one correspondence with the graph limits.
In other words, each kernel $W$ defines a graph limit $\gG_W$, every graph
limit can be represented by a kernel in this way, and two kernels define the same
graph limit if and only if they are equivalent.
Thus, the cut metric defines the same notion of equivalence as the one 
mentioned in the introduction.
Furthermore, $W_1$ and $W_2$ are equivalent if and only if $\dl(W_1,W_2)=0$, see
e.g.\ \cite{SJ249}.

Every kernel is equivalent to a kernel on $\oi$, so it suffices to
consider such kernels. (We shall not use this restiction in the present
paper, however.)

One manifestation of the connection between graph limits and kernels is the
following:
If $G$ is a graph with vertices labelled 1,2,\dots,$n$, let
$A_G(i,j)\=\ett{i\sim j}$ define its adjacency matrix, and let
\begin{equation*}
  W_G(x,y)\=A_G\bigpar{\ceil{nx},\ceil{ny}}.
\end{equation*}
This defines a kernel $W_G$ on $\oi$ (or rather on $(0,1]$, which is equivalent).
A sequence of graphs with $|\gn|\to\infty$ converges to the graph limit
$\gG=\gG_W$ if and only if $\dcut(W_{\gn},W)\to0$.

Note that $W_G$ depends on the labelling of the vertices of $G$, but only in
a rather trivial way, and different labellings yield equivalent kernels.
Here, in the study of monotone kernels, the ordering is relevant.
If $G$ is a graph with a given order $\prec$ on $V$, we therefore
define $W_G=W_{G,\prec}$ as above, but using the labelling of the vertices 
with $1\prec2\prec\cdots$, ignoring the original labelling, if any.

\section{Further measures of quasimonotonicity}\label{SB}

In \refS{S:intro} we defined a functional $\gOa$ that measures,
in an averaged sense, how far the adjacency matrix of a graph is from being monotone.
There are several natural variations of the definition; we shall concentrate on
two.

Firstly, in \eqref{go1<a1} and \eqref{go1<a2}, we were careful to exclude $v$ and $w$
from the set $A$; this had the advantage of making $\gOa(G)$ exactly zero when $G$
is a threshold graph. But most of the time it is more convenient not to do this.
Instead, we consider
\begin{equation}\label{go3<a}
  \gOb(G,\prec,A) \=
\frac1{n^3}\sum_{v\prec w}\bigpar{e(v,A) -e(w,A)}_+,
\end{equation}
which differs from \eqref{go1<a2} in that we count all edges
into $A$, and not just the edges into $A\setminus\set{v,w}$. This changes
each edge count by at most 1, so
\begin{equation}\label{goi}
| \gOa(G,\prec,A)-  \gOb(G,\prec,A)| <1/n.
\end{equation}

As in \eqref{go1<} and \eqref{go1}, we set
\begin{align}
  \gOb(G,\prec) &\=\max_{A\subseteq V}   \gOb(G,\prec,A), \text{\ \ and}
\label{goj<}
\\
  \gOb(G) &\=\min_{\prec}   \gOb(G,\prec).
\label{goj}
\end{align}

Before turning to our second variant, let us note a basic property of $\gOa$.
Let $\be(v,A)$ denote the number of edges from $v$ to $A$ in the complement
$G^\ccc$ of $G$. If $v\notin A$, then $\be(v,A)=|A|-e(v,A)$.
Hence, for any $v$, $w$ and $A$,
\begin{equation*}
 \be(w,A\setminus\set{v,w})-\be(v,A\setminus\set{v,w}) = 
 e(v,A\setminus\set{v,w})-e(w,A\setminus\set{v,w}).
\end{equation*}
From \eqref{go1<a2} it follows that $\gOa(G^\ccc,\succ,A)=\gOa(G,\prec,A)$,
where, naturally, $\succ$ denotes the reverse of the order $\prec$.
Thus $\gOa(G^\ccc,\succ)=\gOa(G,\prec)$ and $\gOa(G^\ccc)=\gOa(G)$.

For $\gOb$ one can show similarly, or deduce using \eqref{goi}, that $|\gOb(G^\ccc)-\gOb(G)|\le 2/n$,
say.

Despite the above symmetry property of $\gOa$,
the following `locally symmetrized'
version of the definition turns out to have technical advantages.
Given a graph $G$, an order $\prec$ on $V(G)$, and $A\subseteq V(G)$, set
\begin{equation}\label{goj'}
 \gOc(G,\prec,A)\=\gOb(G,\prec,A)+\gOb(G,\prec,V\setminus A),
\end{equation}
\begin{equation}\label{goj'a}
 \gOc(G,\prec)\=\max_{A\subseteq V} \gOc(G,\prec,A)
\end{equation}
and
\begin{equation}\label{goj'b}
  \gOc(G)\=\min_\prec \gOc(G,\prec).
\end{equation}
Of course, we could define a corresponding symmetrization of $\gOa$, but we shall not bother.

It is easily seen that all our functionals $\gO_j$ take values in 
$[0,1]$ (in  fact, in $[0,\frac12)$).
%$[0,\frac12)$.
We have the following relations.

\begin{lemma}\label{LB1}
If\/ $G$ is a graph with $|G|=n$, then
\begin{equation}\label{lb1a}
| \gOa(G)-  \gOb(G)| <1/n, 
\end{equation}
and
\begin{equation}\label{goj'd}
  \gOb(G)\le\gOc(G)\le2\gOb(G).
\end{equation}
Consequently, if $\gnn$ is a sequence of graphs with $|\gn|\to\infty$, then
$\gO_j(\gn)\to0$ for some $j$ if and only if this holds for all $j=0,1,2$.
\end{lemma}

\begin{proof}
The inequality \eqref{lb1a} is immediate from \eqref{goi}.

The definition \eqref{goj'} implies that
\begin{equation}\label{goj'c}
  \gOb(G,\prec) \le \gOc(G,\prec) \le 2\gOb(G,\prec),
\end{equation}
which in turn implies \eqref{goj'd}.
\end{proof}

\begin{remark}
Instead of summing in \eqref{go1<a2} or \eqref{go3<a},
in analogy with the standard definition of $\eps$-regular partitions 
(see \eg{} \cite[Section IV.5]{B:MGT}), we may
count the number of `bad' pairs $(v,w)$ of
vertices $v\prec w$ where the difference 
$e(v,A) -e(w,A)$ is larger than $\eps n$, for some small $\eps$.
This suggests the following definition:
with $\prec$ an order on the vertex set $V$, $n\=|V|$, and
$A$ a subset of $V$, set
\begin{equation*}
  \gOd(G,\prec,A)
\=
\inf\Bigset{\eps>0:
 \bigabs{
 \bigset{v\prec w :% \text{ and } 
 e(v,A) >e(w,A)+\eps n
}}
\le\eps n^2},
%\label{go2<a}
\end{equation*}
and define $\gOd(G)$ by taking the maximum over $A$ with $\prec$ fixed, and then
minimizing over $\prec$.
It is a standard observation that if $x_1,\ldots,x_a$ take values in $[0,b]$,
then $\sum_i x_i\ge \eps ab$ implies that there are at least $\eps a/2$ of the $x_i$
that are at least $\eps b/2$, and that if at least $\eps a$ of the $x_i$
are at least $\eps b$, then the sum is at least $\eps^2 ab$.
Using this it is easy to check that $\gOb$ and $\gOd$ are bounded by
suitable functions of each other. In fact, it turns out that
\begin{equation*}
  \tfrac12\gOb(G)\le\gOd(G)\le\gOb(G)\qq.
\end{equation*}
We can also define corresponding modifications of the other $\gO_j$.
\end{remark}

\begin{remark}
\refP{Pgo10} says that
a graph $G$ is a threshold graph if and only if  $\gOa(G)=0$.
This does not hold for $\gOb$; in fact, if $G$ contains an
edge $vw$, with $v\prec w$, then 
$\gOb(G,\prec,\set w)\ge n\qwww e(v,\{w\})=n\qwww$ by
\eqref{go3<a}; hence $\gOb(G)\ge n\qwww$ unless $G$ is empty.
Consequently, $\gOb(G)>0$  for every non-empty graph $G$.
On the other hand, \refP{Pgo10} and \refL{LB1} show that
$\gOb(G)\le 1/n$ for every threshold graph.  
\end{remark}

We defined each $\gO_j(G)$ by
taking the minimum of $\gO_j(G,\prec)$ over all possible orderings $\prec$ of
the vertices. As the next lemma shows, for $\gOc$, ordering the vertices by
their degrees $d(v)\=e(v,V)$ (resolving ties arbitrarily) is optimal.
This is the main reason for considering $\gOc$.

\begin{lemma}\label{LB3X}
Let $<$ be an order on $V$ such that $v<w \implies d(v)\le d(w)$.
Then 
$\gOc(G)= \gOc(G,<)$.
\end{lemma}

\begin{proof}
The inequality $\gOc(G)\le \gOc(G,<)$ is immediate
from the definition \eqref{goj'b}, so it suffices
to prove the reverse inequality.

Let $\prec$ be any order on $V$. If $v<w$, then
$e(v,V)=d(v)\le d(w)=e(w,V)$ and thus,
for
$A\subseteq V$,
\begin{equation}\label{evw}
  \begin{split}
e(v,A)-	e(w,A)
&=
e(v,V)-	e(w,V) + e(w,V\setminus A)-	e(v,V\setminus A)
\\&
\le e(w,V\setminus A)-	e(v,V\setminus A).
  \end{split}
\end{equation}
Let $f(v,w,A)\=\bigpar{e(v,A)-	e(w,A)}_+$ and
$g(v,w,A)\=f(v,w,A)+f(v,w,V\setminus A)$. By \eqref{evw}, if $v<w$,
then
$f(v,w,A)\le f(w,v,V\setminus A)$ and thus
\begin{equation}\label{gvw}
  g(v,w,A)
\le f(w,v,V\setminus A)+f(w,v,A) = g(w,v,A).
\end{equation}

Using \eqref{gvw} for $v<w$ with $v\succ w$, we obtain
\begin{equation*}
  \begin{split}
\gOc(G,<,A)
&\=
\frac1{n^3}\sum_{v< w}g(v,w,A)
\\&
=
\frac1{n^3}\sum_{\substack{v< w\\v\prec w}}g(v,w,A)
+
\frac1{n^3}\sum_{\substack{v< w\\v\succ w}}g(v,w,A)
\\&
\le
\frac1{n^3}\sum_{\substack{v< w\\v\prec w}}g(v,w,A)
+
\frac1{n^3}\sum_{\substack{w> v\\w\prec v}}g(w,v,A)
\\&
=\frac1{n^3}\sum_{v\prec w}g(v,w,A)
= \gOc(G,\prec,A).
  \end{split}
\end{equation*}
Hence, by \eqref{goj'a}, 
$\gOc(G,<)\le \gOc(G,\prec)$. Since $\prec$ is arbitrary,
this yields $\gOc(G,<)=\gOc(G)$.% by \eqref{goj'b}.
\end{proof}

As an immediate consequence of Lemmas \refand{LB3X}{LB1}, we have the following
result for $\gOb$.

\begin{corollary}\label{CLB3}
Let $<$ be an order on $V$ such that $v<w \implies d(v)\le d(w)$.
Then 
$\gOb(G)\le \gOb(G,<)\le 2\gOb(G)$.
\end{corollary}
\begin{proof}
By \eqref{goj'c}, \refL{LB3X} and \eqref{goj'd},
\begin{equation*}
 \gOb(G) \le \gOb(G,<) \le \gOc(G,<) = \gOc(G) \le 2\gOb(G).
\end{equation*}
(Alternatively, one can use a simplified version of the proof of \refL{LB3X}.)
\end{proof}

Using a symmetrized version of $\gOa$, or otherwise, it is easy to prove the 
corresponding 
result for $\gOa$.

\begin{remark}
  \label{RBreg}
If $G$ is regular, then any order $<$ satisfies the condition
of \refL{LB3X} and \refC{CLB3}, so these results show that
$\gOc(G,<)$ is the same for all orders, and
$\gOb(G,<)$ is the same for all orders  within a factor of 2;
the latter holds also for $\gOa$.
\end{remark}

The factor 2 in \refC{CLB3} is
annoying but not really harmful for our purposes.
It is best possible, as shown
by the following example.

\begin{example}\label{EB}
Consider a balanced complete bipartite graph $G=K_{m,m}$ (so $n=2m$), with
bipartition $(V_1,V_2)$.
Given an order $\prec$ on the vertex set $V_1\cup V_2$, let 
$N_{ij}\=\bigabs{\bigset{(x,y)\in V_i\times V_j: x\prec y}}$.
Note that
\begin{equation}  \label{n12+n21}
N_{12}+N_{21}=\bigabs{V_1\times V_2}=m^2.
\end{equation}

Let $A\subseteq V=V_1\cup V_2$ and let $a_i=|A\cap V_i|$, $i=1,2$.
Then $e(v,A)=a_2$ if $v\in V_1$ and $e(v,A)=a_1$ if $v\in V_2$.
Hence,
\begin{equation}\label{eb+}
  \begin{split}
n^3\gOb(G,\prec,A)
&=\sum_{v\prec w}\bigpar{e(v,A)-e(w,A)}_+	
\\&
=N_{12}(a_2-a_1)_++N_{21}(a_1-a_2)_+.
  \end{split}
\end{equation}
Since $a_1$ and $a_2$ can be freely chosen in \set{0,\dots,m}, we have
$a_1-a_2\in\set{-m,\dots,m}$, and maximizing over $A$ yields
\begin{equation}\label{mN}
  n^3\gOb(G,\prec)=m\max\{N_{12},N_{21}\}.
\end{equation}
If $\prec_1$ is an order with all elements of $V_1$ coming first,
then $N_{12}=m^2$ and $N_{21}=0$, and thus
\begin{equation*}
  \gOb(G,\prec_1)=m^3/n^3=1/8.
\end{equation*}
On the other hand, if $m$ is even and $\prec_2$ is an order which starts
with $m/2$
elements of $V_1$, continues with all of $V_2$, and finishes with the remaining half of $V_1$,
then $N_{12}=N_{21}=m^2/2$, and thus
\begin{equation}\label{1/16}
  \gOb(G,\prec_2)=\tfrac12m^3/n^3=1/16.
\end{equation}
Thus $\gOb(G,\prec_1)=2\gOb(G,\prec_2)$ although $G$ is regular and
\refC{CLB3} applies to every order.

For $\gOa$, the ratio between $\gOa(G,\prec_1)$ and $\gOa(G,\prec_2)$
is $2-O(1/n)$ by \eqref{goi}.

Note that for any order $\prec$, \eqref{n12+n21} implies
$\max\set{N_{12},N_{21}}\ge m^2/2$, and
thus \eqref{mN} yields
\begin{equation}\label{mN2all}
 \gOb(G) \ge n^{-3}m^3/2 = 1/16.
\end{equation}
Consequently, if $m$ is even, then \eqref{1/16} shows that
\begin{equation}
  \label{mN2even}
%\gOb(K_{m,m})=
\gOb(G)=\gOb(G,\prec_2)=1/16
\qquad\text{($m$ even)}.
\end{equation}

On the other hand,
if $m$ is 
odd, then since $N_{12}+N_{21}=m^2$ is odd,
for any order $\prec$ we have $\max\{N_{12},N_{21}\}\ge (m^2+1)/2$,
and this is attained for some $\prec$.
Thus \eqref{mN} now yields
\begin{equation}\label{mN2odd}
 \gOb(G) = n^{-3}m(m^2+1)/2 > 1/16
\qquad\text{($m$ odd)}.
\end{equation}
We thus have
\begin{equation}\label{mN2both}
 \begin{cases}
 \gOb(K_{m,m})
 =  1/16,  & \text{$m$ even},\\
 \gOb(K_{m,m})
=(1+m^{-2})/16
>  1/16,  & \text{$m$ odd}.
 \end{cases}
\end{equation}

For $\gOc$, the situation is simpler.
It follows from \eqref{eb+} that
$n^3\gOb(G,\prec,\allowbreak V\setminus A)=N_{12}(a_2-a_1)_-+N_{21}(a_1-a_2)_-$,
and thus, using \eqref{n12+n21}, 
\begin{equation}
 n^3\gOc(G,\prec, A)
=N_{12}|a_2-a_1|+N_{21}|a_1-a_2|
=m^2|a_1-a_2|.
\end{equation}
Maximizing over $A$ we find $\gOc(G,\prec)=m^3/n^3=1/8$ for every order
$\prec$,
\cf{} \refR{RBreg},
and thus $\gOc(G)=1/8$.

If we modify $G$ by adding a perfect matching inside $V_2$ (assuming $m$ is even)
then every order $<$ satisfying the condition of \refC{CLB3} is of the type
$\prec_1$. The added edges change each $e(v,A)$ by at most 1, and thus each
$\gO_j(G,\prec,A)$ is changed by at most $1/n$. Hence this yields an
example where $\gO_j(G,<)=(2-O(1/n))\gO_j(G)$ for $j=0,1$, for every order $<$ considered in
\refC{CLB3}. 
\end{example}

\section{Monotone kernels and graph limits}\label{Smono}

We begin by extending the definition of monotone kernels to other
probability spaces.

\begin{definition}
  An % (linearly) 
\emph{ordered probability space} $(\sss,\prec )=(\sss,\cF,\mu,\prec )$
  is a  probability space $(\sss,\cF,\mu)$ with a (linear) order $\prec $ that
  is   measurable, \ie, \set{(x,y):x\prec y} is a measurable subset of
  $\sss\times\sss$. 
\end{definition}
Note that it follows that
\set{(x,y):x\succ y} and \set{(x,y):x= y} are measurable.

All orders considered in this paper are assumed to be measurable, even if we
only sometimes say so explicitly. Similarly, we only consider
subsets and functions that are  measurable.

The standard example of an ordered probability space
is $\oi$ with Lebes\-gue measure and the standard
order. $\oi$ is always equipped with these unless we say otherwise.

\begin{definition}
Let $(\sss,\prec)$ be an ordered probability space.
A \emph{monotone kernel} on $(\sss,\prec)$ is a kernel $W:\sssq\to\oi$ such that
\begin{equation}\label{q22}
W(x_1,y_1)\le W(x_2,y_2)
\qquad \text{if }
 x_1\preceq x_2,\,
 y_1\preceq y_2.
\end{equation}
\end{definition}

Let $\cwm(\sss,\prec)$ be the set of monotone kernels on $(\sss,\prec)$,
noting that $\cwm=\cwm(\oi)$. 
We shall prove the following properties of $\cwm(\sss,\prec)$ in
Sections \ref{Spfmono1} and \ref{Spfmono2}.

\begin{theorem}\label{Tcwm}
Let $\sssp$ be an
ordered
probability space.
  \begin{thmenumerate}
	\item\label{tcwm1}
$\cwm(\sss,\prec)$ is a compact subset of $L^1(\sssq)$.
\item\label{tcwm2}
Two kernels in $\cwm(\sss,\prec)$ are equivalent if and only if they are \aex{}
equal. 
\item\label{tcwm3}
The metrics $\normll{W_1-W_2}$, $\dl(W_1,W_2)$, $\cn{W_1-W_2}$,
and $\dcut(W_1,W_2)$ 
are equivalent on $\cwm(\sss,\prec)$, i.e., induce the same topology.
  \end{thmenumerate}
\end{theorem}

Recall that $\cum$ denotes the set of monotone graph limits, i.e., the
class of graph limits that can be represented as $\gG_W$ for some $W\in\cwm=\cwm(\oi)$.

\begin{corollary}\label{Ccum}
Each monotone graph limit has a
representation as $\gG_W$ for some
$W\in\cwm=\cwm(\oi)$ with $W$
unique up to equality a.e.
Furthermore, there is a homeomorphism between $\cum$ and $\cwm(\oi)$,
regarded as 
a subset of $\lioi$.
\end{corollary}
\begin{proof}
Immediate from \refT{Tcwm} and the fact that the
  metric on the set of graph limits is equivalent to $\dcut$ on the
  corresponding kernels.
\end{proof}

In \refS{S:intro} we defined $\cum$ as the set of graph limits
that can be represented by some $W\in\cwm(\oi)$. The following theorem shows
that we may allow monotone kernels on arbitrary ordered probability spaces
without changing $\cum$,
\ie, 
$$
\cum=\set{\gG: \exists\, \sssp \text{ and } W\in\cwm(\sss,\prec) 
\text{ such that }\gG=\gG_W}.
%\cum=\set{\gG: \gG=\gG_W \text{ for some $W\in\cwm(\sss)$ for some }\sss}.
$$
This version of the definition is perhaps more natural than considering
$\oi$ only; on the other hand, it is often convenient to use $\oi$.

\begin{theorem}\label{Tmono}
Let $\sssp$ be an ordered probability space, and let $W\in\cwm(\sss,\prec)$.
Then there
is a monotone kernel $W'\in\cwm(\oi)$ that is equivalent to $W$.
Equivalently, $\gG_W\in\cum$.
\end{theorem}

We shall next define two quantitative measures of how far a kernel
is from being monotone,
in analogy with \eqref{go1<a1}--\eqref{go1} (or, more closely,
\eqref{go3<a}, \eqref{goj<} and \eqref{goj}),
and \eqref{goj'}--\eqref{goj'b}.

Given $W\in \liss$, 
a (measurable) order $\prec$ on $\sss$,
and a (measurable) subset $A$ of $\sss$, set
\begin{equation}\label{gowa}
\begin{split}
  \gOkb(W,&\prec,A) \= \\
&\iint_{x\prec  y}\lrpar{\int_A W(x,z)\dd\mu(z)-\int_A W(y,z)\dd\mu(z)}_+
\dd\mu(x)\dd\mu(y),
\end{split}
\end{equation}

\begin{equation}\label{gock}
 \gOkc(W,\prec,A) \=   \gOkb(W,\prec,A) +  \gOkb(W,\prec,\sss\setminus A),
\end{equation}
and, for $j=1,2$,
\begin{align}
  \gO_j(W,\prec)&\=\sup_{A\subseteq\sss}\gO_j(W,\prec,A), \label{ngow<}\\
  \gO_j(W)&\=\inf_{\prec}\gO_j(W,\prec),\label{ngow}
\end{align}
where the infimum is over all measurable orders on $\sss$.
Note that
\begin{equation}\label{gobck}
 \gOkb(W)\le \gOkc(W) \le 2\gOkb(W).
\end{equation}

For $A\subseteq \sss$, let 
$\wa(x)\=\int_A W(x,z)\dd\mu(z)$. 
Then \eqref{gowa} can be written as
\begin{equation}\label{gowa1}
  \gOkb(W,\prec,A) 
=	
\iint_{x\prec  y}\Bigpar{\wa(x)-\wa(y)}_+ \dd\mu(x)\dd\mu(y).
\end{equation}

\begin{remark}
  \label{Rcutt}
It is easily seen that
\begin{equation}\label{rcutt}
\gOkb(W,\prec)
=
\sup_{f,g} \iiint_{x\prec  y}\bigpar{W(x,z)- W(y,z)}
f(x,y)g(z)\dd\mu(x)\dd\mu(y)\dd\mu(z),
\end{equation}
where the supremum is taken over all $f:\sssq\to\setoi$ and $g:\sss\to\setoi$, 
and that allowing all $f:\sssq\to\oi$ and $g:\sss\to\oi$ yields the same
result.
Thus $\gOkb(W,\prec)$ can be seen as a one-sided
version of the cut norm of
the function $\bigpar{W(x,z)- W(y,z)}\etta_{\set{x\prec y}}$ on
$\sssq\times\sss$. 

Similarly, $\gOkc(W,\prec)$ equals
\begin{multline}\label{rcuttkc}
%\gOkc(W,\prec)=
\sup_{f_1,f_2,g} \iiint_{x\prec  y}\bigpar{W(x,z)- W(y,z)}
\bigpar{f_1(x,y)g(z)+f_2(x,y)(1-g(z))}
\\
\cdot
\dd\mu(x)\dd\mu(y)\dd\mu(z)
,
\end{multline}
where the supremum is taken either over all $f_1,f_2:\sssq\to\setoi$ and 
$g:\sss\to\setoi$, or over all $f_1,f_2:\sssq\to\oi$ and $g:\sss\to\oi$.
\end{remark}

In the light of \eqref{gobck}, $\gOkb$ and $\gOkc$ are essentially
equivalent. 
In particular $\gOkb(W)=0\iff\gOkc(W)=0$.
When the difference is not important, we simply write
$\gO$; formally, this may be read as $\gOkb$.  Occasionally, there are
advantages to considering one or the other variant.

\begin{theorem}\label{TM}
  Let $(\sss,\prec)$ be an ordered probability space and let $W$ be a kernel on
  $\sss$. Then $\gO(W,\prec)=0$ if and only if $W$ is \aex{} equal to
  a monotone kernel.
\end{theorem}

As noted above, $\gO_j$, $j=1,2$, is an analogue of $\gO_j$ defined earlier
for graphs. Indeed, there is a simple relation.

\begin{lemma}\label{LD2}
If\/ $G$ is a graph with an order $\prec$ on the vertex set $V$, 
and $<$ denotes the standard order on $\oi$,
then
$\gO_j(W_G,<)=\gO_j(G,\prec)$ for $j=1,2$.
\end{lemma}

For $\gOkc$, we shall show that \refL{LD2} implies a corresponding
result after minimizing over the relevant orderings.
\begin{lemma}\label{LD3}
If\/ $G$ is a graph, then $\gOkc(W_G)=\gOkc(G)$.
\end{lemma}

Note that $W_G$ depends on the labelling of the vertices in $G$, 
but this is
harmless since the different versions differ by \mpp{} bijections of $\oi$
(in fact, permutations of subintervals) and  obviously have the same
$\gO_j(W_G)$.

\begin{remark}\label{RD3}
Let $G=K_{m,m}$ as in \refE{EB}. Then $W_G$ does not
depend on $m$, and one can check that $\gOkb(W_G)=1/16$.
For $m$ odd, we have $\gOb(G)>1/16$ by \eqref{mN2odd}.
Thus we can have $\gOkb(W_G)<\gOkb(G)$.
It seems likely that the difference is bounded by some function
tending to 0 as $n\to\infty$, but we have not proved anything
stronger than $\gOkb(W_G)\le \gOkb(G)\le 2\gOkb(W_G)$,
which follows from \refL{LD3} and the relationship
between $\gO_1$ and $\gO_2$.
\end{remark}

\begin{remark}\label{RD3b} 
Given a graph $G$,
define $\wv_G$ as the adjacency matrix of $G$, regarded as a kernel on
$V=V(G)$, which we
regard as a probability space with the uniform
probability measure (each point has mass $1/|G|$).
It is easily verified that $\gOkb(\wv_G,\prec,A)=\gOb(G,\prec,A)$
for every order $\prec$ on $V$ and every set $A\subseteq V$. Hence
$\gOkb(\wv_G,\prec)=\gOb(G,\prec)$ for every order $\prec$ and
$\gOkb(\wv_G)=\gOb(G)$, and the same holds for $\gO_2$.

Note that $\wv_G$ and $W_G$ are equivalent
kernels. It follows from \refL{LD3} that $\gOkc(\wv_G)=\gOkc(W_G)$,
but \refR{RD3} shows that $\gOkb(\wv_G)>\gOkb(W_G)$ if $G=K_{m,m}$ with $m$
odd.
(See also \refC{Cjeppe} and \refR{Rjeppe} below.)
\end{remark}

\begin{remark}
  \label{RK}
In \eqref{ngow}, we take the infimum over all measurable orders on $\sss$. In
general, this may be problematic, since there are \ps{s} with no
measurable orders, see \refE{EK} below. In such cases, we interpret
\eqref{ngow} as $\gO_j(W)=\infty$ (or perhaps 1), but this has the unhappy
consequence that two equivalent kernels $W_1$ and $W_2$ may have
$\gO_2(W_1)\neq\gO_2(W_2)$. For example, let $W_1$ and $W_2$ both be constant
$1/2$, with $W_1$ defined on $\oi$  and $W_2$ on a space
$\sss$ with no measurable order; then $\gO_2(W_1)=0$ and $\gO_2(W_2)=\infty$.
In the sequel we therefore consider only $\sss$ that have at least one
measurable order. Even in this case, equivalent
kernels may have different $\gOkb$; see \refR{RD3b}.
We will show in \refC{Cjeppe} that there is no such problem for $\gOkc$.
The case $\gO(W)=0$ is covered by the following theorem.
\end{remark}

\begin{theorem}
  \label{TM2}
Let $W$ be a kernel on a \ps{} $\sss$ with at least one measurable order.
Then the following are equivalent. 
\begin{romenumerate}
\item \label{tm2-gO}
$\gO(W)=0$. 
\item \label{tm2-mono}
There exists a measurable order $\prec$ on $\sss$ such that\/
$W$ is \aex{} equal to a monotone kernel on $\sssp$.
\item \label{tm2-equiv}
$W$ is equivalent to a monotone kernel on some \ops.
\item \label{tm2-oi}
$W$ is equivalent to a monotone kernel on  $\oi$.
\item \label{tm2-gG}
$\gG_W$ is a monotone graph limit.
\end{romenumerate}
\end{theorem}

\begin{example}  \label{EK}
Let $\sss=\oi$, but equipped with the
$\gs$-field $\cF_0$ consisting of the subsets of $\sss$ that are either
countable or have a countable complement. For the measure $\mu$ we take the
restriction of the Lebesgue measure to $\cF$. (Thus, $\mu(A)=0$ if $A$ is
countable, and $\mu(A)=1$ otherwise.)

Let $\cC$ be the family of countable subsets of $\sss$.
The $\gs$-field $\cF\times\cF$ is contained in the $\gs$-field
$$
\bigset{A\subseteq\sssq:\exists B_1,B_2\in\cC \text{ such that }
A\text{ or }\sss\setminus A\subseteq(B_1\times \sss)\cup(\sss\times B_2)
}.
$$
Thus, if $\prec$ is a measurable order, then there exist $B_1,B_2\in\cC$
such that either
\begin{align*}
\set{(x,y):x\prec y}&\subseteq(B_1\times \sss)\cup(\sss\times B_2)
\intertext{or}
\set{(x,y):x\succeq y}&\subseteq(B_1\times \sss)\cup(\sss\times B_2);
\end{align*}
in the latter case we have
\begin{align*}
\set{(x,y):x\prec y}
\subset
\set{(x,y):x\preceq y}&\subseteq(B_2\times \sss)\cup(\sss\times B_1).
\end{align*}
However, in both cases we find that if we choose two distinct
$x,y\notin(B_1\cup B_2)$, then neither $x\prec y$ nor $y\prec x$ holds, which
is a contradiction. Thus $(\sss,\cF,\mu)$ is a probability space
supporting no measurable orders.
\end{example}

\section{Proofs of Theorems \ref{Tcwm}--\ref{Tmono}}\label{Spfmono1}

A \emph{downset} in an ordered set $(\sss,\prec)$ is a subset $A$ such that if
$x\prec y$ and $y\in A$, then $x\in A$.
We begin with two lemmas concerning simple (and certainly well-known)
properties of downsets; for completeness we give full proofs.

\begin{lemma}
  \label{Lo0}
  \begin{thmenumerate}
\item
If\/ $A$ and $B$ are downsets in a linearly ordered set $\sssp$,
then $A\subseteq B$ or $B\subseteq A$.
\item
If $A$ and $B$ are downsets in an ordered probability space $\sssp$ with $\mu(A)<\mu(B)$, then
$A\subset B$.	
  \end{thmenumerate}
\end{lemma}

\begin{proof}
\pfitem{i}
Otherwise there would exist $x\in A\setminus B$ and $y\in B\setminus A$,
but then neither $x\prec y$, $y\prec x$ nor $x=y$ is possible.
  \pfitem{ii}
%Since $\mu(A)<\mu(B)$, 
Now  $B\subseteq A$ is impossible, and the result follows by (i).
\end{proof}

\begin{lemma}
  \label{Lo}
If\/ $\sssp$ is an ordered probability space without atoms, then for every
$t\in\oi$ there exists a downset $D(t)$ with $\mu(D(t))=t$. 
Furthermore, $D(t)\subset D(u)$ when $t<u$.
\end{lemma}

\begin{proof}
It suffices to prove the first statement; the second then follows by
\refL{Lo0}(ii). 

For $x\in \sss$, let $D_x$ be the 
downset \set{y\in\sss:y\preceq x}.
Let $X=X_0,X_1,X_2,\dots$ be an \iid{} sequence of random points in $\sss$
(with the distribution $\mu$). Since there are no atoms, $\P(X_i=X_j)=0$ for
all $i\neq j$. Thus, for every $n$, $X_0,\dots,X_n$ are \as{} distinct, and
by symmetry, all $(n+1)!$ orderings of them have the same probability
$1/(n+1)!$. Hence,
\begin{equation*}
  \E\bigpar{\mu(D_X)^n}
=\P(X_1,\dots,X_n\prec X_0)=\frac{n!}{(n+1)!}=\frac1{n+1},
\qquad n\ge1.
\end{equation*}
Consequently, $\mu(D_X)$ has the same moments as the uniform distribution
$U(0,1)$, and thus $\mu(D_X)\sim U(0,1)$.

It follows that the set \set{\mu(D_x):x\in\sss} is a dense subset of \oi.
Hence, for every $t\in(0,1]$, there exists a sequence $(x_i)_i$ in $\sss$
such that $\mu(D_{x_i})\upto t$ as $\itoo$. 
Then $D_{x_i}\subset D_{x_{i+1}}$ for $i\ge1$ by \refL{Lo0}(ii), and we
can take $D(t)\=\bigcup_{i=1}^\infty D_{x_i}$, which is a downset with 
$\mu(D(t))\=\lim_{\itoo}\mu(D_{x_i})=t$. For $t=0$ we take $D(0)\=\emptyset$.
\end{proof}

Given an integrable function $W$ on $\sssq$ and $A,B\subseteq\sss$ with
$\mu(A),\mu(B)>0$, let
\begin{equation}\label{wab}
  \bW(A,B)\=\frac1{\mu(A)\mu(B)}\iint_{A\times B}W(x,y)\dd\mu(x)\dd\mu(y)
\end{equation}
denote the average of $W$ over $A\times B$.
If $\cP=\set{ A_i}$ is a finite partition of $\sss$, 
we say that a function on $\sssq$ is a \emph{$\cP$-step function} if it is
constant on each set $A_i\times A_j$. 
(A \emph{step function} on $\sssq$ is a $\cP$-step function for some finite
partition $\cP$.)
If $W\in \liss$,
we let
$W_{\cP}$ be the $\cP$-step function defined by 
\begin{equation}\label{wp}
  W_{\cP}(x,y)=\bW(A_i,A_j) \quad\text{for}\quad x\in A_i,\, y\in A_j.
\end{equation}
If some $A_i$ has measure 0, then $W_{\cP}$ is not defined everywhere,
but it is
always defined \aex, which suffices for us. Note that $W_{\cP}$ is the
conditional expectation of $W$ given the $\gs$-field $\cF_\cP\times\cF_\cP$,
where $\cF_\cP$ is the finite $\gs$-field on $\sss$ generated by $\cP$.
It follows that $\cutnorm{W_{\cP}}\le\cutnorm{W}$ and
$\normll{W_{\cP}}\le\normll{W}$.
If $W$ is a kernel, then $W_{\cP}$ is also a kernel. 
A kernel that is also a step function,
such as $W_{\cP}$, 
is called a \emph{step kernel}.

Suppose now that $\sssmp$ is an atomless \ops, and let $D(t)$, $0\le
t\le1$, be an 
increasing family of downsets in $\sss$ with $\mu(D(t))=t$ as in \refL{Lo},
with $D(0)=\emptyset$ and $D(1)=\sss$.

For $n\ge1$ and $i=1,\dots,n$, define
\begin{equation}\label{Ani}
 A_i=A_{ni}\=D(i/n)\setminus D((i-1)/n).
\end{equation}
Then $\cP_n\=\set{A_{ni}}_i$ is a partition of $\sss$ into $n$ sets of the same
measure $1/n$. Furthermore, if $i< j$, then $A_{ni}\prec A_{nj}$, meaning that
if $x\in A_{ni}$ and $y\in A_{nj}$, then $x\prec y$.

Given a kernel $W$ on $\sss$,
let $\wn{ij}\=\bW(A_{ni},A_{nj})$ and let $W_n$ be the step kernel
$W_{\cP_n}$; thus $W_n=\wn{ij}$ on $A_{ni}\times A_{nj}$.
Define the step kernels $W_n^\pm$ by
$W_n^+(x,y)\=\wn{i+1,j+1}$ and 
$W_n^-(x,y)\=\wn{i-1,j-1}$ on $A_{ni}\times A_{nj}$, where $\wn{ij}=0$ if
$i$ or $j=0$ and $\wn{ij}=1$ if $i$ or $j=n+1$.
%(The case $i=0$, $j=n+1$, or vice versa, is undefined; it will not occur.)

If $W$ is monotone, then the matrix $(\wn{ij})_{ij}$ is increasing along
each row and column, and thus $W_n$ is a monotone step kernel.

\begin{lemma}
  \label{LV2}
Let $W$ be a monotone kernel on an atomless \ops{} $\sssp$. Then
$W_n^-\le W\le W_n^+$,
$W_n^-\le W_n\le W_n^+$ and
\begin{equation*}
  \norm{W_n-W}\qliss
\le\norm{W_n^+-W_n^-}\qliss \le 4/n.
\end{equation*}
\end{lemma}

\begin{proof}
If $(x,y)\in A_{ni}\times A_{nj}$ and $(x',y')\in A_{n,i+1}\times A_{n,j+1}$
(with $i,j\le n-1$), then $W(x,y)\le W(x',y')$, and averaging over $(x',y')$
it follows that
$W(x,y)\le\wn{i+1,j+1}=W_n^+(x,y)$. 
This  inequality evidently holds also if $i$ or $j=n$.
Hence $W\le W_n^+$.
Similarly, $W\ge W_n^-$.

Averaging over each $A_{ni}\times A_{nj}$, it follows that
$W_n^-\le W_n\le W_n^+$. (This also follows directly
from the monotonicity of
  $\wn{ij}$.) Consequently,
$|W_n-W|\le W_n^+-W_n^-$, and thus
\begin{equation*}
  \begin{split}
\norm{W_n-W}\qliss
&\le\iint_\sssq\lrpar{W_n^+-W_n^-}
=\iint_\sssq W_n^+ -\iint_\sssq W_n^-
\\&
=n\qww\sum_{i,j=2}^{n+1}\wn{ij}-n\qww\sum_{i,j=0}^{n-1}\wn{ij}
\le 2n\qww\sum_{i=n}^{n+1}\sum_{j=2}^{n+1}\wn{ij}
\\&
\le \xfrac4n.
  \end{split}\qedhere
\end{equation*}
\end{proof}

Trivially, for any kernel $W$ we have $\cn{W}\le \norm{W}_\liss$. In general
there is no reverse inequality. However, if $\cP$ is a partition of $\sss$ into $n$ sets
and $W$ is a $\cP$-step function, then it is trivial to bound $\norm{W}_\liss$ from above
by a polynomial times $\cn{W}$. Indeed, one can write $\norm{W}_\liss$
as a sum of $n$ integrals of the form in \eqref{cutnorm}, in each taking $g$ to be $1$ on
one part of $\cP$ and zero elsewhere, and choosing the sign of $f$ on each part appropriately.
In fact, the correct polynomial order is $\sqrt{n}$, as shown in \cite{SJ249}.

\begin{lemma}\label{LW1new}
Let $\sss$ be a probability space and $\cP$ a partition of $\sss$ into $n$ sets.
If $W$ is a $\cP$-step function, then $\norm{W}\qliss\le \sqrt{2n}\cn{W}$.
Furthermore, for any $W\in\liss$ we have
\begin{equation}\label{lw1}
\norm{W_{\cP}}\qliss\le \sqrt{2n}\cn{W}.
\end{equation}
\end{lemma}
\begin{proof}
It suffices to prove the first statement; the second follows immediately,
since $W_{\cP}$ is a $\cP$-step function, and $\cn{W_{\cP}}\le\cn{W}$.

The statement and proof are (essentially) present in Remark 9.8 of
\cite{SJ249}. Nevertheless, 
let us write out the proof.

In 1930, Littlewood~\cite{Littlewood} proved that there is a constant
$c\le \sqrt{3}$ such that for any $n$-by-$n$ array of real numbers $a_{ij}$ we have
\begin{eqnarray*}
  \sumin \Bigpar{\sumjn|a_{ij}|^2}\qq
&\le& c \max_{\eps_i,\eps_j'=\pm1} \sumin\sumjn\eps_i\eps_j'a_{ij}\\
&=& c  \max_{\eps_i=\pm1} \sumjn \biggabs{\sumin \eps_ia_{ij}}
    =c  \max_{\eps_j=\pm1} \sumin \biggabs{\sumjn \eps_ja_{ij}}.
\end{eqnarray*}
Later it was noticed (see \cite{Zyg}, Ch. 5 and \cite{Blei}) that
this inequality of Littlewood's could be deduced from a
special case of an inequality that had been proved some years earlier by
Khintchine~\cite{Khin}. In 1976, Szarek~\cite{Szar} proved that the best constant
in Littlewood's inequality (in fact, in the corresponding inequality of Khintchine)
is $\sqrt{2}$. For some related results, see, e.g., \cite{FHJSZ}, \cite{Haa1},
\cite{Haa2}, \cite{KoKw} and \cite{Lat}.

As noted in~\cite{SJ249}, using the Cauchy--Schwartz inequality and
Littlewood's inequality, with the constant $c=\sqrt{2}$ proved by Szarek,
it follows that
\begin{equation}\label{littlewood2}
  \sumin \sumjn|a_{ij}|
\le
  \sumin n\qq \Bigpar{\sumjn|a_{ij}|^2}\qq
\le \sqrt{2n}\,\max_{\eps_i,\eps'_j=\pm1} \sumin\sumjn\eps_i\eps'_ja_{ij}.
\end{equation}

Returning to the proof of Lemma~\ref{LW1new},
let the parts of $\cP$ be $A_1,\ldots,A_n$, and set $a_{ij}=\mu(A_i)\mu(A_j)W_{ij}$,
where $W_{ij}$ is the value of $W$ on $A_i\times A_j$.
Then $\norm{W}_{\liss}=\sum_{ij}|a_{ij}|$. In the definition \eqref{cutnorm} of the
cut norm, restricting our attention to functions $f,g:\sss\to\{\pm1\}$
that are constant on each $A_i$, we find that
\begin{equation}\nonumber
 \cn{W}\ge \max_{\eps_i,\eps'_j=\pm1} \sumin\sumjn\eps_i\eps'_ja_{ij}
\end{equation}
(in fact, equality holds), so the result follows from \eqref{littlewood2}.
\end{proof}

As noted in~\cite{SJ249}, it is easy to check that the factor $\sqrt{2n}$ is best possible
apart from the constant, for example by considering $0/1$-valued kernels associated to
random graphs.
For arbitrary {\em monotone} kernels, the lemmas above allow
us to bound the $L^1$-norm in terms of the cut norm.

\begin{theorem}\label{TV3}
If\/ $W_1$ and $W_2$ are monotone kernels on an \ops{} $\sssp$, then
\begin{equation}\label{lv3}
\norm{W_1-W_2}\qliss\le 10\cn{W_1-W_2}^{2/3}.
\end{equation}
\end{theorem}

\begin{proof}
  Suppose first that $\sss$ is atomless.
Let $n\ge1$ and consider the partition $\cP_n=\set{A_{ni}}_i$ defined in \eqref{Ani}
and the step
kernels $W_{k,n}=(W_k)_{\cP_n}$, $k=1,2$. 
\refL{LW1new} yields
\begin{equation}\label{v3a}
\norm{W_{1,n}-W_{2,n}}\qliss
=\norm{(W_1-W_2)_{\cP_n}}\qliss
\le \sqrt{2n}\cn{W_1-W_2}.
\end{equation}
By \refL{LV2}, we have $\norm{W_{k}-W_{k,n}}\qliss\le 4/n$, so
by the triangle inequality
\begin{equation*}
  \norm{W_1-W_2}\qliss
\le \norm{W_{1,n}-W_{2,n}}\qliss + 8/n
\le \sqrt{2n}\cn{W_1-W_2} +8/n.
\end{equation*}
The result for atomless $\sss$ now follows by choosing 
$n\=\bigceil{\cn{W_1-W_2}^{-2/3}}\le 2\cn{W_1-W_2}^{-2/3}$.
(In the case $\cn{W_1-W_2}=0$, we let \ntoo.)

If $\sss$ has atoms, we consider the atomless \ps{} $\hsss\=\sss\times\oi$
with the lexicographic order.
Let $\pi:\hsss\to\sss$ be the projection onto the first coordinate
and let $\hW_k\=W_k^\pi$ be the
extension of $W_k$ to $\hsss$.
The proof just given applies to $\hsss$, and thus
\begin{equation*}
  \norm{W_1-W_2}_{L^1(\sss^2)}
=  \norm{\hW_1-\hW_2}_{L^1(\hsss^2)}
\le   10\cn{\hW_1-\hW_2}^{2/3}
=
 10\cn{W_1-W_2}^{2/3}.
\end{equation*}
\end{proof}

\begin{example}\label{2/3best}
It is easy to see that \eqref{lv3} is tight apart from the constant. Indeed,
let $\sss$ be the discrete probability space with $n$ equiprobable elements $\{0,1,\ldots,n-1\}$,
and choose two $0/1$-valued kernels on $\sss$ with $\norm{W_1-W_2}_\liss=\Theta(1)$
and $\cn{W_1-W_2}=\Theta(n^{-1/2})$. For example, we may take kernels corresponding
to two independent instances of the random graph $G(n,1/2)$. 
Let $W$ be the function defined by $W(i,j)=i+j$. Then it is easy to see that
$W_i'=(W_i+W)/(2n)$ is a monotone kernel for each $i$.
Since $\norm{W_1'-W_2'}_\liss=\norm{W_1-W_2}_\liss/(2n)=\Theta(n^{-1})$ and
$\cn{W_1'-W_2'}=\cn{W_1-W_2}/(2n)=\Theta(n^{-3/2})$, this gives {\em monotone} kernels
$W_1'$ and $W_2'$ with $\norm{W_1'-W_2'}_\liss=\Theta(\cn{W_1'-W_2'}^{2/3})$.
\end{example}

Our next aim is to prove the rather unsurprising fact that if we start from two monotone
kernels, then `rearranging' one or both does not bring them any closer
in the $L^1$ distance. First we need a preparatory lemma; this can be viewed
as a continuous, coupling version of the trivial observation
that if we wish to minimize $\sum_{i=1}^n |a_i-b_i|$ (or, equivalently, $\sum (a_i-b_i)_+$)
where the values in each sequence are given but we are allowed to permute them,
then we should sort both sequences into ascending order.

\begin{lemma}\label{LW4}
If\/ $h_1,h_2:\sss\to\bbR$ are increasing integrable functions on an \ops\/
$\sssmp$, and 
$\gf_1,\gf_2:\sss'\to\sss$ are \mpp{} maps from a \ps{} $(\sss',\mu')$ to
$(\sss,\mu)$, then
\begin{equation}\label{iplus}
 \int_{\sss'} (h_1^{\gf_1}-h_2^{\gf_2})_+\dd\mu' \ge \int_\sss (h_1-h_2)_+\dd\mu
\end{equation}
and $\norm{h_1^{\gf_1}-h_2^{\gf_2}}_{L^1(\sss')}\ge\norm{h_1-h_2}_{L^1(\sss)}$. 
\end{lemma}

\begin{proof}
For any integrable function on any measure space we have
$\norm{h}_{L^1}=\int (h)_+ + \int (-h)_+$, so it suffices to prove the first statement.

For any function $f$ and real number $t$, let $B_f(t)\=\set{x:f(x)\le t}$.
Fubini's theorem yields
\begin{equation*}
  \begin{split}
	\int_\sss(h_1-h_2)_+\dd\mu
&=\int_\sss\intoooo\ett{h_1(x)>t\ge h_2(x)}\dd t\dd\mu(x)
\\&
=\intoooo\int_\sss\ett{x\in B_{h_2}(t)\setminus B_{h_1}(t)}\dd\mu(x)\dd t
\\&
=\intoooo\mu\bigpar{B_{h_2}(t)\setminus B_{h_1}(t)}\dd t.
  \end{split}
\end{equation*}
Similarly,
\begin{equation}\nonumber
\int_{\sss'}(h_1^{\gf_1}-h_2^{\gf_2})_+ \dd\mu'
=\intoooo\mu'\bigpar{B_{h_2^{\gf_2}}(t)\setminus B_{h_1^{\gf_1}}(t)}\dd t.
\end{equation}
Since the $\gf_i$ are measure preserving,
we have $\mu'\bigpar{B_{h_i^{\gf_i}}(t)}=\mu'\bigpar{\gf_i\qw(B_{h_i}(t))}=\mu(B_{h_i}(t))$.
Since $h_1$ and $h_2$ are increasing, $B_{h_1}(t)$ and $B_{h_2}(t)$ are downsets,
so by \refL{Lo0} they are nested. The result follows by noting
that $\mu(X\setminus Y)\ge (\mu(X)-\mu(Y))_+$, with equality if $X$ and $Y$ are nested.
\end{proof}

\begin{lemma}\label{LW6}
  If\/ $W_1$ and $W_2$ are monotone kernels on an \ops{} $\sssp$, then
$\dl(W_1,W_2)=\norm{W_1-W_2}\qliss$.
\end{lemma}

\begin{proof}
  Suppose that $\gf_1,\gf_2$ are \mpp{} maps $\sss'\to\sss$ for some \ps{}
  $(\sss',\mu')$.
Then, using \refL{LW4} on each coordinate separately,
\begin{equation*}
  \begin{split}
&	{\norm{W_1^{\gf_1}-W_2^{\gf_2}}_{L^1((\sss')^2)}}
\\
&=
\int_{\sss'}\int_{\sss'}\bigabs{W_1(\gf_1(x),\gf_1(y))-W_2(\gf_2(x),\gf_2(y))}
\dd\mu'(x)\dd\mu'(y)
\\&
\ge \int_{\sss'}\int_{\sss}\bigabs{W_1(t,\gf_1(y))-W_2(t,\gf_2(y))}
\dd\mu(t)\dd\mu'(y)
\\&
\ge \int_{\sss}\int_{\sss}\bigabs{W_1(t,u)-W_2(t,u)}
\dd\mu(t)\dd\mu(u)
=\norm{W_1-W_2}\qliss,
  \end{split}
\end{equation*}
where for the last step we first apply Fubini's Theorem to change the order 
of integration.
The result follows by the definition \eqref{dl}.
\end{proof}

With a little more work, we obtain a corresponding result for the cut norm and cut metric.
Unfortunately, we need to consider a variant of the definition.

If $W$ is an integrable function on $\sssq$, let
\begin{equation}\label{cutnorm01}
 \cnone W \= \sup_{f,g:\sss\to\{0,1\}}
  \Bigabs{\int_{\sss^2} W(x,y)f(x)g(y)\dd\mu(x)\dd\mu(y)},
\end{equation}
where the supremum is over all pairs of measurable $0/1$-valued functions on $\sss$.
(We could equally well consider functions taking values in $[0,1]$; the value
of the supremum does not change.) Expressing each of the functions $f,g$ in \eqref{cutnormpm}
as the difference of two $0/1$-valued functions, we see that
\begin{equation}\label{c1c2}
 \cnone W \le \cn W \le 4\cnone W,
\end{equation}
so for all questions concerning convergence, the norms are equivalent.

In analogy with \eqref{dcut}, given $W_i\in L^1(\sss_i^2)$, $i=1,2$, let
\begin{equation}\label{dcut01}
\dcutone(W_1,W_2)
\=\inf_{\gf_1,\gf_2}\cnone{W_1^{\gf_1}-W_2^{\gf_2}},
\end{equation}
where, as in \eqref{dcut}, the infimum is taken over all couplings $(\gf_1,\gf_2)$ of $\sss_1$ and
$\sss_2$.

\begin{lemma}\label{LW6cut1}
  If\/ $W_1$ and $W_2$ are monotone kernels on an \ops{} $\sssp$, then
$\dcutone(W_1,W_2)=\cnone{W_1-W_2}$.
\end{lemma}

\begin{proof}
Suppose that $\gf_1,\gf_2$ are \mpp{} maps $\sss'\to\sss$ for some \ps{} $\sss'$.
It suffices to show that $\cnone{W_1^{\gf_1}-W_2^{\gf_2}}\ge \cnone{W_1-W_2}$.

Given a probability space $(\sss,\mu)$, an integrable function $W$ on $\sss^2$,
and two functions $f,g:\sss\to\{0,1\}$, set
\begin{equation*}
 I_{f,g}(W) \= \int_{\sss^2} W(x,y)f(x)g(y)\dd\mu(x)\dd\mu(y),
\end{equation*}
so $\cnone{W} = \sup_{f,g} |I_{f,g}(W)|$. Swapping $W_1$ and $W_2$ if necessary,
we may assume that $\cnone{W_1-W_2} = \sup_{f,g}I_{f,g}(W_1-W_2)$.
Hence, fixing (arbitrary) functions $f,g:\sss\to\{0,1\}$, it suffices to prove
that
\begin{equation}\label{oaim}
 \sup_{f',g'} I_{f',g'}(W_1^{\gf_1}-W_2^{\gf_2}) \ge I_{f,g}(W_1-W_2),
\end{equation}
since $\cn{W_1^{\gf_1}-W_2^{\gf_2}}$ is at least the left-hand side.

The first statement \eqref{iplus} of \refL{LW4} says exactly that
if $h_1$ and $h_2$ are increasing, integrable functions on $\sssmp$
and $\gf_1$, $\gf_2:(\sss',\mu')\to(\sss,\mu)$ are
measure-preserving, then
\begin{equation}\label{ff'}
\begin{split}
 \max_{f':\sss'\to\{0,1\}} \int_{\sss'} \bigpar{h_1(\gf_1(x))-h_2(\gf_2(x))}f'(x)\dd\mu'(x) \\
  \ge  \max_{f:\sss\to\{0,1\}} \int_\sss \bigpar{h_1(t)-h_2(t)}f(t)\dd\mu(t),
\end{split}
\end{equation}
where the maximization is over all $\{0,1\}$-valued functions on the
relevant space; the corresponding supremum is clearly attained.
We shall use this inequality twice; in particular, we shall
twice use the observation that a specific $f$ on the right
is `beaten' by some $f'$ on the left.

Let $h_i(t)=\int_\sss W_i(t,u)g(u)\dd\mu(u)$. Then (since $g(u)$ is
non-negative), 
$h_i$ is monotone. Applying (the observation following) \eqref{ff'}
to these functions and our function $f$, we find that there is some
$f':\sss'\to \{0,1\}$ such that
\begin{equation*}
 \begin{split}
&\int_{\sss'} \left(\int_\sss \left(W_1(\gf_1(x),u)-W_2(\gf_2(x),u)\right)g(u)\dd\mu(u)\right) f'(x)\dd\mu'(x)
\\
&\ge 
\int_{\sss} \left(\int_\sss \left(W_1(t,u)-W_2(t,u)\right)g(u)\dd\mu(u)\right) f(t)\dd\mu(t) = I_{f,g}(W_1-W_2).
\end{split}
\end{equation*}
Using Fubini's Theorem, we may rewrite the left-hand side as
\begin{equation*}
 I\= \int_\sss \left(\int_{\sss'} \bigpar{W_1(\gf_1(x),u)-W_2(\gf_2(x),u)}f'(x)\dd\mu'(x)\right)g(u)\dd\mu(u).
\end{equation*}
Let $h_i'(u)=\int_{\sss'}W_i(\gf_i(x),u)f'(x)\dd\mu'(x)$. Then the $h_i'$ are again monotone,
so applying \eqref{ff'} to these functions and $g$ gives a $g':\sss'\to\{0,1\}$ such that
\begin{equation*}
 \int_{\sss'} \left(\int_{\sss'} \bigpar{W_1(\gf_1(x),\gf_1(y))-W_2(\gf_2(x),\gf_2(y))}f'(x)\dd\mu'(x)\right)g'(y)\dd\mu'(y) \ge I.
\end{equation*}
But now the left-hand side is simply $I_{f',g'}(W_1^{\gf_1}-W_2^{\gf_2})$,
so we have
$I_{f',g'}(W_1^{\gf_1}-W_2^{\gf_2}) \ge I \ge I_{f,g}(W_1-W_2)$, establishing \eqref{oaim}.
\end{proof}

In the light of \eqref{c1c2}, \refL{LW6cut1} has the following immediate corollary.
\begin{lemma}\label{LW6cut}
  If\/ $W_1$ and $W_2$ are monotone kernels on an \ops{} $\sssp$, then
$\dcut(W_1,W_2) \ge \cn{W_1-W_2}/4$.\noproof
\end{lemma}
It seems plausible that $\dcut(W_1,W_2)=\cn{W_1-W_2}$ for monotone kernels, but we
do not have a proof (or indeed a strong feeling that this is actually true).

We are now ready to bound the $L^1$ distance with `rearrangement' in terms of the cut metric,
when the kernels in question are monotone.

\begin{lemma}\label{LW7}
If\/ $W_1$ and $W_2$ are monotone kernels on an \ops{} $(\sss,\mu,\prec)$, then
\begin{equation}\label{lw7}
\dl(W_1,W_2)\le
26\, \dcut(W_1,W_2)^{2/3}.
\end{equation}
\end{lemma}

\begin{proof}
Combining \refL{LW6}, \refT{TV3} and \refL{LW6cut}, we have
\begin{equation}\label{dldc}
 \dl(W_1,W_2) = \norm{W_1-W_2}\qliss
 \le 10\cn{W_1-W_2}^{2/3}
 \le 10(4\dcut(W_1,W_2))^{2/3},
\end{equation}
giving the result.
\end{proof}

\begin{remark}
Using \refT{Tmono} (which is proved below), 
\refL{LW7} immediately extends
to   monotone kernels defined on possibly different \ops{s}.
\end{remark}

\begin{remark}\label{RV3}
The exponent $2/3$ in \eqref{lw7} is best possible, as shown by the
kernels $W_1'$, $W_2'$ in \refE{2/3best}.
Indeed, for these kernels, the first inequality in \eqref{dldc} is tight
up to the constant. The second inequality is always tight up to the constant
$4^{2/3}$ since, by definition, $\dcut(W_1,W_2)\le\cn{W_1-W_2}$.
\end{remark}

We are now ready to prove the first few results in \refS{Smono}.

\begin{proof}[Proof of \refT{Tcwm}]
The equivalence of the different metrics in \ref{tcwm3} follows from
\refT{TV3}, Lemmas \ref{LW6} and \ref{LW6cut} (see \eqref{dldc}) and the
inequality $\dcut(W_1,W_2)\le \dl(W_1,W_2)$.

As a special case, for two kernels $W_1,W_2\in\cwm(\sss)$,
$$\dcut(W_1,W_2)=0\iff \norm{W_1-W_2}\qliss=0 \iff W_1=W_2 \text{ \aex}, 
$$
which establishes \ref{tcwm2}.

For \ref{tcwm1}, we show that $\cwm(\sss)$ is closed and totally bounded
as a subset of $L^1(\sssq)$.
First, if $W_\nu\in\cwm(\sss)$ and $W_\nu\to W$ in $L^1(\sssq)$ as \nutoo, 
then there is a subsequence that converges \aex{} to $W$, and replacing $W$
by the $\limsup$ of that subsequence, we see that $W\in\cwm(\sss)$.
Hence, $\cwm(\sss)$ is closed.

Next,
first assume that $\sss$ is atomless. By \refL{LV2}, for every $n$
there is a partition $\cP_n$ such that for every kernel
$W\in\cwm(\sss)$, 
there is a $\cP_n$-step kernel $W_n$ with
$\norm{W-W_{n}}\qliss\le4/n$. If $F_n$ is the finite set of $\cP_n$-step
kernels taking values in 
\set{0,\frac1n,\frac2n,\dots,1},
%\set{i/n:i=0,1,\dots,n}, 
then there always exists a
$W_n'\in F_n$ with 
$\norm{W_n-W'_n}\qliss\le 1/n$, and thus 
$\norm{W-W'_n}\qliss\le 5/n$. Since $n$ is arbitrary, this shows that
$\cwm(\sss)$ 
is totally bounded.

If $\sss$ has atoms, we consider as above $\hsss=\sss\times\oi$ and
$\pi:\hsss\to\sss$; then 
$W\mapsto W^\pi$ is an isometric embedding of $L^1(\sssq)$ into
$L^1(\hsss^2)$. This embeds $\cwm(\sss)$ into $\cwm(\hsss)$, and since the
latter is totally bounded,  $\cwm(\sss)$ is too.
\end{proof}

\begin{proof}[Proof of \refT{Tmono}]
If $\sss$ has atoms, we replace it, as above, by $\hsss=\sss\times\oi$; thus
we may assume that $\sss$ is atomless. By \refL{LV2}, there is a sequence of
step kernels $W_n$ that converges to $W$ in $L^1(\sssq)$. Each $W_n$ is
obviously equivalent to the monotone step kernel $W'_n$ on $\oi$ defined by
$W_n'=\wn{ij}$ on $I_i\times I_j$, where $I_i\=((i-1)/n,i/n]$.
We have $\norm{W_n'-W_m'}\qlioi=\norm{W_n-W_m}\qliss$, and thus $(W_n')$ is
a Cauchy sequence in $\lioi$. Hence there is some $W'$ such that $W_n'\to W'$ in $\lioi$,
and \refT{Tcwm}\ref{tcwm1} implies that $W'\in\cwm(\oi)$.
For every $n$,
\begin{equation*}
\begin{split}
\dcut(W,W')&\le\dcut(W,W_n)+\dcut(W_n,W_n')+\dcut(W_n',W') \\
&\le \frac4n+0+\norm{W_n'-W'}_{L^1([0,1]^2)}.
\end{split}
\end{equation*}
Since $W_n'\to W'$ in $\lioi$, it
follows that $\dcut(W,W')=0$, so $W'$ and $W$ are equivalent.
\end{proof}

\section{Proofs of Theorems \ref{TM}--\ref{TM2}}\label{Spfmono2}

%We begin with some simple results.

In this section we prove the remaining results in \refS{Smono},
namely, \refT{TM}, Lemmas \ref{LD2} and \ref{LD3}, and \refT{TM2}.

We start with a technical lemma, which is fairly obvious but nevertheless
deserves to be stated precisely.

\begin{lemma}\label{LX}
Suppose that $(\sss_1,\mu_1,\prec_1)$ and $(\sss_2,\mu_2,\prec_2)$ are
\ops{s}, and that $\sssqq$ is equipped with a probability measure $\mu$ such
that the projection $\pi_1$ onto $\sss_1$ is \mpp. Let $\precx_1$ be the
lexicographic order on $\sssqq$. If $W$ is a kernel on $\sss_1$, then for $j=1,2$,
\begin{equation*}
  \gO_j(W,\prec_1)=\gO_j(W^{\pi_1},\precx_1).
\end{equation*}
\end{lemma}

In most applications, we take $\mu=\mu_1\times\mu_2$.

\begin{proof}
Writing $x\in \sss\=\sssqq$ as $x=(x_1,x_2)$,
by \eqref{rcutt}, $\gOkb(W^{\pi_1},\precx_1)$ is equal to
\begin{equation}\label{qk}
%\gO(W,\prec)
\sup_{f,g} \iiint_{x\precx_1  y}\bigpar{W(x_1,z_1)- W(y_1,z_1)}
f(x,y)g(z)\dd\mu(x)\dd\mu(y)\dd\mu(z),
\end{equation}
where the supremum is over all $f:\sssq\to\oi$ and $g:\sss\to\oi$.

Let $\cF_1$ be the $\gs$-field on $\sss$ obtained by pulling back that
on $\sss_1$. Thus the
$\cF_1$-measurable functions are all functions
of the form $h(x_1,x_2)=h_1(x_1)$ for
measurable $h_1$ on $\sss_1$.
In \eqref{qk} we may replace $f$ and $g$ by their conditional expectations
given $\cF_1\times\cF_1$ and $\cF_1$, respectively.
Recalling that $\precx_1$ is lexicographic, and noting that
the integrand vanishes when $x_1=y_1$, \eqref{qk} reduces to
\begin{equation*}
\sup_{f_1,g_1} \iiint_{x_1\prec_1  y_1}\bigpar{W(x_1,z_1)- W(y_1,z_1)}
f_1(x_1,y_1)g_1(z_1)\dd\mu_1(x_1)\dd\mu_1(y_1)\dd\mu_1(z_1),
\end{equation*}
with the supremum over $f_1:\sss_1^2\to\oi$ and $g_1:\sss_1\to\oi$.
By \eqref{rcutt}, this is simply $\gOkb(W,\prec_1)$.

(In the special case when $\mu=\mu_1\times\mu_2$, the argument above
is equivalent to simply integrating over $x_2,y_2,z_2$ in \eqref{qk}.) 

For $\gOkc$, the argument is similar, using \eqref{rcuttkc} instead of \eqref{rcutt}.
\end{proof}

\begin{proof}[Proof of \refT{TM}]
Here it makes no difference whether we consider $\gOkb$ or $\gOkc$, so we
simply write $\gO$. 

If $W=W'$ \aex{} where $W'$ is monotone, then we have
$\gO(W,\prec,A)=\gO(W',\prec,A)=0$ for all $A\subseteq\sss$, and hence $\gO(W,\prec)=0$.

Conversely, suppose that $\gO(W,\prec)=0$.
Let $A,B,C,D\subseteq\sss$
have positive measures,
and suppose that $A\prec B$. Since
$\gO(W,\prec)=0$, we have $\gO(W,\prec,C)=0$ and thus by \eqref{gowa1}
$\wx C(x)\le \wx C(y)$ for \aex{} $(x,y)$ with $x\prec y$, and in particular
for \aex{} $(x,y)\in A\times B$.
Averaging over all such $(x,y)$ yields
$\bW(A,C)\le\bW(B,C)$. Similarly, by symmetry, if $C\prec D$, then
$\bW(B,C)\le\bW(B,D)$. Consequently,
letting $A\preceq B$ mean $A\prec B$ or $A=B$,
\begin{equation}\label{gfw}
\bW(A,C)\le\bW(B,D)
\qquad\text{if $A\preceq B$, $C\preceq D$}.
\end{equation}

Assuming still that $A,B,C,D\subseteq\sss$ have positive measures,
suppose that $A\prec B$ and $C\prec D$. If $A_1\subseteq A$ and $C_1\subseteq C$,
then \eqref{gfw}, applied to $A_1,B,C_1,D$, yields
\begin{equation*}
  \iint_{A_1\times C_1} W
\le (\mu\times\mu)(A_1\times C_1)\bW(B,D).
\end{equation*}
Since every measurable subset of $A\times C$ can be approximated (in
measure) by a finite disjoint union of rectangle sets $A_i\times C_i$, and
$W$ is bounded, it follows that
\begin{equation*}
  \iint_{E} W
\le (\mu\times\mu)(E)\bW(B,D)
\quad\text{for every } E\subseteq A\times C.
\end{equation*}
Taking $E\=\set{(x,y)\in A\times C: W(x,y)>\bW(B,D)}$, we obtain
$\mu\times\mu(E)=0$, and thus
\begin{equation}\label{kia1}
  W(x,y)\le\bW(B,D) 
\quad\text{\aex{} on $A\times C$ when $A\prec B$ and $C\prec D$}.
\end{equation}
Similarly, by reversing the inequalities,
\begin{equation}\label{kia2}
  W(x,y)\ge\bW(B,D) 
\quad\text{\aex{} on $A\times C$ when $A\succ B$ and $C\succ D$}.
\end{equation}

Suppose now that $\sss$ is atomless, and consider, for a given $n$, 
the partition $\cP=(A_i)_1^n$ defined in \eqref{Ani}.
By \eqref{gfw}, $W_n\=W_{\cP}$ is a monotone kernel.
By \eqref{kia1} and \eqref{kia2}, 
$W_n^-(x,y) \le W(x,y)\le W_n^+(x,y)$ \aex{} on each $A_i\times A_j$, and
thus \aex{} on $\sssq$. Further, by averaging this or directly from
\eqref{gfw},  also $W_n^-\le W_n\le W_n^+$. It follows as in the proof of
\refL{LV2} that 
\begin{equation}\label{cec}
 \norm{W_n-W}\qliss\le4/n. 
\end{equation}

Now consider the sequence $W_{2^k}$, $k\ge1$. By \eqref{cec} and the
Borel--Cantelli lemma, or by the martingale convergence theorem, 
$W_{2^k}\to W$ \aex{} as $k\to\infty$. Hence, if we define 
$W'\=\limsup_{k\to\infty} W_{2^k}$, then $W=W'$ \aex{} and $W'$ is a
monotone kernel. This completes the proof when $\sss$ is atomless.

If $\sss$ has atoms, we may either modify the argument above, or use our
standard trick of replacing $\sss$ by $\sss\times\oi$, using \refL{LX}; this
gives a monotone
kernel $W'$ on $\sss\times\oi$ with $W'((x,a),(y,b))=W(x,y)$ for \aex{}
$(x,a,y,b)\in(\sss\times\oi)^2$, 
and thus $W$ is \aex{} equal to the monotone kernel $W''$ on $\sss$ defined by
$W''(x,y)=\intoi\intoi W'((x,a),(y,b))\dd a\dd b$.
\end{proof}

\begin{proof}[Proof of \refL{LD2}] 
Let $I_i\=((i-1)/n,i/n]$ and for $A\subseteq\oi$, set $A_i\=A\cap
  I_i$. For $j=1,2$, by \eqref{gowa} and \eqref{gock}, $\gO_j(W_G,<,A)$ depends only on the numbers
  $a_i\=\mu(A_i)\in[0,1/n]$; moreover, since the function $u\mapsto u_+$ is convex,
  $\gO_j(W_G,<,A)$ is a convex function of $(a_1,\dots,a_n)$; hence it attains
  its maximum when each $a_i$ is either $0$ or $1/n$. In other words, it
  suffices to consider $A=\bigcup_{i\in B}I_i$ for some $B\subseteq V$.
In this case, it is easily seen that $\gO_j(W_G,<,A)=\gO_j(G, \prec,B)$,
noting that $\int_A W_G(x,z)\dd z=\int_A W_G(y,z)\dd z$ if $x,y\in I_i$ for
  some $i$.
The result follows by taking the maximum over $B\subseteq V$.
\end{proof}

\begin{lemma}\label{LD1}
Let $(\sss,\prec)$ be an ordered probability space, and let $j\in \{1,2\}$.
  \begin{romenumerate}
\item
If\/ $W_1,W_2\in L^1(\sssq)$, then
\begin{align*}
\gO_j(W_1+W_2,\prec,A) &\le\gO_j(W_1,\prec,A) +\gO_j(W_2,\prec,A) , \\
\gO_j(W_1+W_2,\prec) &\le\gO_j(W_1,\prec) +\gO_j(W_2,\prec) .
\end{align*}
\item
If\/ $W\in L^1(\sssq)$, then 
$\gO_j(W,\prec)\le j \cn{W}$.
\item
If\/ $W_1,W_2\in L^1(\sssq)$, then
$\bigabs{\gO_j(W_1,\prec)-\gO_j(W_2,\prec)}\le j\cn{W_1-W_2}$.
  \end{romenumerate}
\end{lemma}

\begin{proof}
\pfitem{i}
An immediate consequence of the inequality $(a+b)_+\le a_++b_+$ for real $a$
and $b$, 
and the definitions \eqref{gowa}--\eqref{ngow<}.

\pfitem{ii}
By \eqref{gowa1} and Fubini's theorem,
\begin{equation*}
  \begin{split}
  \gOkb(W,\prec,A) 
&
%=	
%\iint_{x\prec  y}\Bigpar{\wa(x)-\wa(y)}_+\dd\mu(x)\dd\mu(y)
\le 	
\iint_{x\prec  y}\Bigpar{\bigabs{\wa(x)}+\bigabs{\wa(y)}}\dd\mu(x)\dd\mu(y)
\\&
=\ints \mu\set{y:y\succ x}\bigabs{\wa(x)}\dd\mu(x) 
 + \ints \mu\set{x:x\prec y}\bigabs{\wa(y)}\dd\mu(y)
\\&
=\ints \mu\set{z:z\neq x}\bigabs{\wa(x)}\dd\mu(x) 
\le
\ints \bigabs{\wa(x)}\dd\mu(x)
\\&
=
\iint_\sssq W(x,y) f(x) g(y)\dd\mu(x)\dd\mu(y)
\le\cutnorm{W},
  \end{split}
\end{equation*}
where $f(x)\=\sign(\wa(x))$ and $g(y)\=\etta_A(y)$; the final inequality
follows from the definition \eqref{cutnorm} of the cut norm.
Now apply \eqref{gock}, if $j=2$, and take the supremum over
$A$.
\pfitem{iii}
A simple consequence of (i), applied to the sums $W_1+(W_2-W_1)$ and $W_2+(W_1-W_2)$, and (ii).
\end{proof}

The function  $\mW=\ws(x)\=\ints W(x,y)\dd\mu(y)$ is known as
the \emph{marginal} of $W$.
(There is also a second marginal, obtained by integrating over
the first variable. Here we consider only symmetric functions, so the
two marginals coincide.)
It is well known that the marginal of a kernel is the natural analogue of
the degree sequence of a graph, see \eg{} \cite{SJ238}. 
We have the following analogue of \refL{LB3X}.

\begin{lemma}\label{LA3}
Let $<$  be a (measurable) order on $\sss$ and assume that
$x<y\implies \mW(x)\le\mW(y)$. 
Then $\gOkc(W,<)=\gOkc(W)$.
%, and $\gOkb(W,<)\le 2\gOkb(W)$.
\end{lemma}

\begin{proof}
Follow the proof of \refL{LB3X}, replacing sums by integrals and degrees
by the values of $\mW$.
% The second statement follows by \eqref{gobck}.
\end{proof}

\begin{remark}
For $\gOkb$, it follows by \eqref{gobck} that $\gOkb(W,<)\le 2\gOkb(W)$.
The factor 2 here is best possible, just as in
\refC{CLB3}. This can be seen by taking $W=W_G$ where $G$ is the complete
bipartite graph $K_{m,m}$ considered in \refE{EB}.
%In other words, $W$ is the
%indicator function of the set 
%$\bigpar{(0,\frac12)\times(\frac12,1)}\cup
%\bigpar{(\frac12,1)\times(0,\frac12)}$.
%Then, by \refL{LD2} or direct calculations, 
%$\gOkb(W,\prec)=1/8$ for the natural order, while $\gOkb(W)=1/16$.
\end{remark}

\begin{corollary}\label{CA3}
Let $\sss$ be a probability space and\/ $W$ a kernel on $\sss$.
Then $\gOkc(W)=0$ if and only if there exists an order $\prec$ on $\sss$ such
that $\gOkc(W,\prec)=0$.
\end{corollary}

\begin{proof}
The `if' direction is clear. Thus, assume $\gOkc(W)=0$.
Then there exists a measurable order $\prec_0$ on $\sss$. Define an order
$\prec$ on $\sss$ by
\begin{equation}\label{ca3}
  x\prec y \quad\text{if}\quad
%  \begin{cases}
\mW(x)<\mW(y)
\text{ or } (\mW(x)=\mW(y) \text{ and } x\prec_0 y).	
%  \end{cases}
\end{equation}
This is a measurable order to which \refL{LA3} applies, so
$\gOkc(W,\prec)=\gOkc(W)=0$.
\end{proof}
Of course, the same result for $\gOkb$ follows by \eqref{gobck}.

\begin{proof}[Proof of \refL{LD3}]
Recall that $W_G=W_{G,\prec}$ depends on a labelling of the vertices of $G$, via
the associated order $\prec$ on $V(G)$. However,
$\gOkc(W_{G,\prec})$ is independent of the order $\prec$.

For any order $\prec$ on $V=V(G)$, \refL{LD2} shows that, using $\prec$ to
define $W_G$, and writing $<$ for the standard order on $[0,1]$, we have
$\gOkc(W_G)\le\gOkc(W_G,<)=\gOc(G,\prec)$. Thus $\gOkc(W_G)\le\gOc(G)$.

Conversely, let $\prec$ be an order on $V$ such that $v\prec w\implies
d(v)\le d(w)$, and use this order to define $W_G$. Then $W_G$ satisfies the
assumption of \refL{LA3} with the standard order $<$ on $\oi$, and thus
$\gOkc(W_G,<)=\gOkc(W_G)$.
Hence, by \refL{LD2} again,
\begin{equation*}
  \gOc(G)\le\gOc(G,\prec)=\gOkc(W_G,<)=\gOkc(W_G).
\qedhere
\end{equation*}
\end{proof}

Our next lemma shows that $\gOkc$ is continuous with respect to the cut metric.

\begin{lemma}\label{Ljeppe}
If\/ $W_1$ and $W_2$ are kernels on \pss{} $\sss_1$ and $\sss_2$, and there
exists a measurable order on $\sss_1$, then 
$\gOkc(W_1)\le \gOkc(W_2)+2\dcut(W_1,W_2)$.
\end{lemma}

\begin{proof}
Recall that the set of step functions is dense in $L^1(\sss_1^2)$. Hence,
for any $\eps>0$, there exists a step kernel $W_1'$ on $\sss_1$ with
$\cn{W_1-W_1'}\le\norm{W_1-W_1'}_{L^1(\sss_1^2)}<\eps$.   
By \refL{LD1}(iii), replacing $W_1$ by $W_1'$ changes $\gOkc(W_1)$ by less
than $2\eps$, and the same holds for $\dcut(W_1,W_2)$. Hence, it suffices to
prove the result when $W_1$ is a  step kernel.

Consequently, assume that $W_1$ is a $\cP$-step kernel, for a finite partition
$\cP=(A_i)_i$ of $\sss_1$. Then its marginal
$\wisi$ is constant on each $A_i$, and we may assume that $A_1,A_2,\dots$ are
labelled such that $\wisi(x)\le\wisi(y)$ if $x\in A_i$, $y\in A_j$ with $i<j$.
Let $\prec_0$ be a measurable order on $\sss_1$, and define $\prec_1$ by
\begin{equation*}
x\prec_1y'\quad\text{if}\quad
x\in A_i\text{ and } y\in A_j \text{ with }
(i<j \text{ or }  (i=j \text{ and } x\prec_0y)).
\end{equation*}

Let $\prec_2$ be any measurable order on $\sss_2$. Consider a coupling
$(\pi_1,\pi_2)$ defined on $(\sssqq,\mu)$ for some $\mu$.
Let $\precx_1$ be the lexicographic order on $\sssqq$, and let $\precx_2$ be
the lexicographic order with the factors in opposite order.
By \refL{LX},
\begin{equation}\label{sw1}
  \gOkc(W_k,\prec_k)=  \gOkc(W_k^{\pi_k},\precx_k),
\qquad k=1,2.
\end{equation}
Moreover, \refL{LA3} applies to $\precx_1$ and $W_1^{\pi_1}$ and shows that
\begin{equation}\label{sw2}
  \gOkc(W_1^{\pi_1},\precx_1)
=  \gOkc(W_1^{\pi_1})
\le  \gOkc(W_1^{\pi_1},\precx_2),
\end{equation}
and by \refL{LD1}(iii),
\begin{equation}\label{sw3}
  \gOkc(W_1^{\pi_1},\precx_2)
\le \gOkc(W_2^{\pi_2},\precx_2)+2\cn{W_1^{\pi_1}-W_2^{\pi_2}}.
\end{equation}
Combining \eqref{sw1}--\eqref{sw3}, we find
\begin{equation*}
  \gOkc(W_1,\prec_1)
\le \gOkc(W_2,\prec_2)+2\cn{W_1^{\pi_1}-W_2^{\pi_2}},
\end{equation*}
and the result follows by taking the infimum over all couplings
such $(\pi_1,\pi_2)$,
\ie, over all probability measures $\mu$ with the right marginals,
and then over all orders $\prec_2$.
\end{proof}

\begin{corollary}\label{Cjeppe}
If\/  $W_1$ and $W_2$ are equivalent kernels on \ps{s} $\sss_1$ and $\sss_2$
that have 
measurable orders, then $\gOkc(W_1)=\gOkc(W_2)$,
and $\frac12\gOkb(W_2)\le\gOkb(W_1)\le2\gOkb(W_2)$.  
\end{corollary}
\begin{proof}
We have $\dcut(W_1,W_2)=0$; the first statement follows by \refL{Ljeppe}.
To deduce the second, use \eqref{gobck}.
\end{proof}

\begin{remark}\label{Rjeppe}
The equivalent of \refL{Ljeppe} for $\gOkb$ does not hold,
and the inequalities $\frac12\gOkb(W_2)\le\gOkb(W_1)\le2\gOkb(W_2)$
in \refC{Cjeppe} are best possible. In fact, if  $W_m\=\wv_{K_{m,m}}$ is the 
kernel defined in \refR{RD3b} for the bipartite graph $K_{m,m}$,
then $W_m$ is equivalent to $W_{K_{m,m}}$ (defined on $\oi$), but
$W_{K_{m,m}}$ is the same for all $m$.
Hence, all $W_m$ are equivalent.
Nevertheless, \refR{RD3b} and \eqref{mN2both} show that 
$\gOkb(W_m)=\gOb(K_{m,m})=(1+m\qww)/16$ if $m$ is odd, while
$\gOkb(W_m)=\gOb(K_{m,m})=1/16$ if $m$ is even.
In particular, $\gOkb(W_1)=1/8=2\gOkb(W_2)$.

On the other hand,
for kernels $W_1,W_2$
on the standard space $\sss=\oi$ (and thus for kernels on any
atomless Borel spaces),  it follows from \eqref{dcut2}  
and \refL{LD1}(iii) that 
%\begin{equation}
 $ \bigabs{\gOkb(W_1)- \gOkb(W_2)}\le \dcut(W_1,W_2)$,
%\end{equation}
since clearly $\gOkb(W_2^{\gf})=\gOkb(W_2)$ for a \mpp{} bijection $\gf$.
In particular, $\gOkb(W_1)=\gOkb(W_2)$ for any two equivalent kernels on $\oi$.
Hence the unruly behaviour of $\gOkb$ is caused by the atoms.
\end{remark}

\begin{proof}[Proof of \refT{TM2}]
\ref{tm2-gO}$\implies$\ref{tm2-mono}.
We use $\gOkc$. If $\gOkc(W)=0$, then by \refC{CA3} there exists an order
$\prec$ on $\sss$ 
such that $\gOkc(W,\prec)=0$, and \refT{TM} shows that $W$ is \aex{} equal to
a monotone kernel on $(\sss,\prec)$.

\ref{tm2-mono}$\implies$\ref{tm2-equiv}. Trivial.

\ref{tm2-equiv}$\implies$\ref{tm2-gO}. 
If $W$ is equivalent to a monotone kernel $W'$ on some \ps{}
$\sss'$, then $\dcut(W,W')=0$ and $\gOkc(W')=0$, and thus $\gOkc(W)=0$ by
\refL{Ljeppe}. 

\ref{tm2-equiv}$\iff$\ref{tm2-oi} $\iff$\ref{tm2-gG}. 
By \refT{Tmono}.
\end{proof}

\section{Proof of Theorems \ref{TQ}--\ref{TQ1}}\label{SpfTQ}

After the preparation above, the proofs are simple.

\begin{proof}[Proof of \refT{TQ1}]
Let $W$ be a kernel on $\oi$ representing $\gG$, \ie, $\gG=\gG_W$ and
$\gn\to W$. Since $\gn\to W$, we have $\dcut(W_{\gn},W)\to0$.

Suppose first that $\gG\in\cum$; we then may choose $W\in\cwm$, and thus
$\gOkc(W,<)=0$ so $\gOkc(W)=0$. Then, by Lemmas \refand{LD3}{Ljeppe},
\begin{equation*}
\gOc(\gn)
= \gOkc(W_{\gn})
\le \gOkc(W)+2\dcut(W_{\gn},W)
= 2\dcut(W_{\gn},W)
\to0.
\end{equation*}
Hence $\gOc(\gn)\to0$, and by \refL{LB1}, $\gOa(\gn)\to0$ as well.

Conversely, suppose that $\gOa(\gn)\to0$, and thus by \refL{LB1}
$\gOc(\gn)\to0$. 
Then, by Lemmas \refand{Ljeppe}{LD3} again,
\begin{equation*}
\gOkc(W)\le \gOkc(W_{\gn})+2\dcut(W_{\gn},W)
= \gOc(G_{\nu})+2\dcut(W_{\gn},W)\to0,
\end{equation*}
and thus $\gOkc(W)=0$. Hence, $\gG=\gG_W\in\cum$ by \refT{TM2}.
\end{proof}

\begin{proof}[Proof of \refT{TQ}]
  If $\gOa(\gn)\to0$, then the same holds for every subsequence. Hence
  \refT{TQ1} shows that every convergent subsequence has a limit that is in
  $\cum$, which by definition says that $(\gn)$ is \qm.

Conversely, suppose that $(\gn)$ is \qm{} but
$\gOa(\gn)\not\to0$. We can then find $\eps>0$ and a subsequence along
which $\gOa(\gn)>\eps$. 
By restricting to a suitable subsubsequence, we may further assume that
$(\gn)$ converges to some limit $\gG$. By the assumption that $\gnn$ is \qm,
$\gG\in\cum$ and thus  
by \refT{TQ1}, $\gOa(\gn)\to0$ along the subsubsequence, a contradiction.
\end{proof}

\section{Quasithreshold graphs}\label{Sqt}

In the definition \eqref{go1<} of $\gOa(G,\prec)$,
we take the maximum over $A$ of the sum in \eqref{go1<a1}. If instead we
take the maximum inside the sum, then we obtain the functional
\begin{equation}\label{gox30}
  \goxa(G,\prec)
\=
\frac1{n^3}\sum_{v\prec w}
\bigabs{N(v)\setminus(N(w)\cup\set w)},
\end{equation}
since $\max_A\bigpar{e(v,A\setminus\set{w}) -e(w,A\setminus\set{v})}_+$ is
obtained by taking (for example) $A=N(v)\setminus N(w)$. 
From $\gOb$, we similarly obtain the slightly simpler functional
\begin{equation}\label{gox31}
  \goxb(G,\prec)
\=
\frac1{n^3}\sum_{v\prec w}
\bigabs{N(v)\setminus N(w)}
=\goxa(G,\prec)+O(1/n).
\end{equation}

For a kernel $W$ on an \ops{} $(\cS,\mu,\prec)$, taking
the supremum over $A$ inside the double integral in \eqref{gowa}, we define
\begin{equation}\label{goxw}
\gox(W,\prec)\=
\iiint_{x\prec y} \bigpar{W(x,z)-W(y,z)}_+\dd\mu(x)\dd\mu(y)\dd\mu(z).
\end{equation}
(Cf.\ \eqref{rcutt}.)
For any graph $G$ with an ordering $\prec$ of the vertices,
corresponding to \refL{LD2} we have
\begin{equation}\label{goxwg}
  \gox(W_G,<)=\goxb(G,\prec).
\end{equation}

Obviously, $\goxa(G,\prec)\ge\gOa(G,\prec)$, and similarly for $\goxb$
and $\gox$.

Let
\begin{align}
  \label{gox}
\gox_j(G)&\=\min_\prec \gox_j(G,\prec) \quad(j=0,1),
&
\gox(W)&\=\inf_\prec \gox(W,\prec).
\end{align}

For kernels, we can use $\gox$ instead of $\gO$ to characterize
monotonicity, \cf{} Theorems \refand{TM}{TM2}.

\begin{theorem}\label{TMx}
  Let $\sssmp$ be an ordered probability space and $W$ a kernel on
  $(\sss,\mu)$. Then $\gox(W,\prec)=0$ if and only if $W$ is \aex{} equal to
  a monotone kernel.
\end{theorem}

\begin{proof}
  If $W$ is \aex{} equal to a monotone kernel, then $W(x,z)\le W(y,z)$ for
  \aex{} $(x,y,z)$ with $x\prec y$, and thus $\gox(W,\prec)=0$. The converse
  follows by \refT{TM}, since $\gOkb(W,\prec)\le\gox(W,\prec)$.
\end{proof}

\begin{theorem}
  \label{TM2x}
Let $W$
be a kernel on a \ps{} $\sss$ with at least one measurable order.
Then $\gox(W)=0$ if and only if
$W$ is \aex{} equal to a monotone kernel on $\sssp$
for some 
order $\prec$ on $\sss$.
\end{theorem}
\refT{TM2} gives further equivalent conditions, for example that
$\gG_W$ is a monotone graph limit.

\begin{proof}
If $\gox(W)=0$, then $\gOkb(W)=0$,
since $\gOkb(W)\le\gox(W)$. Hence the conclusion follows by \refT{TM2}.

Conversely, if $W$ is \aex{} equal to a monotone kernel om $\sssp$, then
$\gox(W)\le\gox(W,\prec)=0$ by \refT{TMx}.  
\end{proof}

For a sequence of graphs, we cannot replace $\gOa$ by $\goxa$ in
\refT{TQ}. In fact, we have the following result, which shows that
$\goxa(\gn)\to0$ characterizes {\em threshold} graph limits rather than
monotone graph limits.
(Recall that threshold graph limits are the monotone
graph limits that correspond to  \oivalued{} kernels; see \refR{RLSz}.)

As usual, we define the \emph{edit distance} $\de(G,G')$ of two graphs on the same
vertex set $V(G)=V(G')$ by $\de(G,G')=\abs{E(G)\setdiff E(G')}$.
If $\cA$ is a class of graphs, then
\begin{equation}\label{dea}
  \de(G,\cA)\=\inf\bigset{\de(G,G'): G'\in\cA\text{ and } V(G')=V(G)}.
\end{equation}

\begin{theorem}
  \label{TT}
Let $(\gn)$ be a sequence of graphs with $|\gn|\to\infty$. 
Then the following are equivalent.
\begin{romenumerate}
  \item\label{tt1}
$\goxa(\gn)\to0$.
\item\label{ttt}
Every convergent subsequence of $\gnn$
has a limit that is a threshold graph limit.
\item\label{ttd}
$\de(\gn,\cT)=o\bigpar{\abs{\gn}^2}$, where $\cT$ is the class of threshold
  graphs. 
\item\label{tte}
There exists a sequence of threshold graphs $\gn'$ with $V(\gn')=V(\gn)$ and 
$\bigabs{E(\gn)\setdiff E(\gn')}=o\bigpar{\abs{\gn}^2}$.
\item\label{ttl1}
There exists a sequence of threshold graphs $\gn'$ with $V(\gn')=V(\gn)$ and 
$\norm{W_{\gn}-W_{\gn'}}\qliss=o(1)$.
\item\label{ttcn}
There exists a sequence of threshold graphs $\gn'$ with $V(\gn')=V(\gn)$ and 
$\cutnorm{W_{\gn}-W_{\gn'}}=o(1)$.
\end{romenumerate}
\end{theorem}

We say that a sequence  $(\gn)$ of graphs with $\abs{\gn}\to\infty$ 
is \emph{quasithreshold} if it satisfies one, and thus all, of the
conditions in \refT{TT}.

As a special case of the equivalence (i)$\iff$(ii), we see that if
$\gn\to\gG$, then $\gG$ is a threshold graph limit if and only if
$\goxa(\gn)\to0$; \cf{} \refT{TQ1}.

The proof of \refT{TT} is simpler than the proof of \refT{TQ}, but we will
nevertheless need some other results first.
One complication is that there is no analogue of \refL{LD1}(iii); as is
shown by the following example, $\gox(W,\prec)$ is not continuous for the 
cut norm.

\begin{example}
  \label{Ediscont}
Let $W=1/2$ be constant on $\oii$, and let $\gnxx$ be a sequence of graphs
with $|\gnx|=n$ and $\gnx\to W$, \ie, $\gnxx$ is a sequence of quasirandom graphs.
(E.g., let $\gnx$ be random graphs $G(n,1/2)$.)
Then, for every $\eps>0$,
$\bigabs{\abs{N(v)\setminus N(w)}-n/4}\le\eps n$ for all but $o(n^2)$ pairs
$(v,w)\in V_{\gnx}^2$, and thus for any order $\prec$,
$\bigabs{n^3\goxb(\gnx,\prec)-n^3/8}\le\eps n^3+o(n^3)$,
so
$\bigabs{\goxb(\gnx,\prec)-1/8}\le\eps+o(1)$.
Since $\eps$ is arbitrary, it follows that
\begin{equation*}
  \gox(W_{\gnx})=\goxb(\gnx)\to\tfrac18
\neq0=\gox(W),
\end{equation*}
although $\cutnorm{W_{\gnx}-W}\to0$.
\end{example}

$\gox$ is obviously continuous in the stronger $L^1$ norm.
It is possible to prove \refT{TT} using this fact and \refL{LT3} below, 
but it is simpler to use
another extension of $\goxb$ to kernels.

\begin{definition}
  If $(\sss,\mu)$ is an atomless probability space and $\prec$ an order on
  $\sss$, let
  \begin{equation}\label{goxx}
\goxx(W,\prec)\=\iiint_{x\prec y}
	W(x,z)\bigpar{1-W(y,z)}\dd\mu(x)\dd\mu(y)\dd\mu(z). 
  \end{equation}
If $\sss$ has atoms, we add half the integral over $x=y$ (and any $z$), i.e.,
%$\frac12\iint_{x=y}	\dots %W(x,z)\bigpar{1-W(y,z)}\dd\mu(x)\dd\mu(y)\dd\mu(z). 
we add $\frac12\iint W(x,z)\bigpar{1-W(x,z)}\mu\set x\dd\mu(x)\dd\mu(z)$.
\end{definition}

The definition in the case that $\sss$ has atoms is such that
$\goxx(W,\prec)=\goxx(\hW,\hprec)$, where $\hW$ is the extension of $W$
to the atomless probability space $\hsss\=\sss\times \oi$ and $\hprec$ is
the lexicographic order on $\hsss$.

Note that if $W$ is \oivalued, then $\goxx(W,\prec)=\gox(W,\prec)$. In
particular, for any graph with an order $\prec$ on $V=V(G)$, by \eqref{goxwg},
\begin{equation}\label{goxxg}
  \goxb(G,\prec)=\gox(W_G,<)=\goxx(W_G,<).
\end{equation}

For our purposes $\goxx$ is better than $\gox$  in two different ways.
The first is that, unlike $\gox$, $\goxx$ is continuous with respect to the cut norm. 
Before proving this, we recall a basic property of the cut norm.
(See \eg{} \cite{SJ249} for a proof.)

\begin{lemma}\label{L0}
If\/ $W\in\liss$, then $\norm{\mW}\qlis\le\cutnorm{W}$.
\noproof
\end{lemma}

%\begin{proof}  
%If  $f\in L^\infty(\sss)$, then
%\begin{equation*}
%  \int_\sss\mW(x)f(x)\dd\mu(x)
%=\int_{\sss^2} W(x,y)f(x)\dd\mu(x)\dd\mu(y)
%\end{equation*}
%and the result follows from \eqref{cutnorm}, 
%letting $g(y)=1$ and
%taking taking $f(x)$ equal to the sign of $\mW(x)$.
%\end{proof}

Recall that, by definition, a kernel $W$ takes values in $[0,1]$.

\begin{lemma}\label{LT1}
Let $(\sss,\prec)$ be an ordered probability space.
If\/ $W_1$ and $W_2$ are kernels on $\sss$, then
$\bigabs{\goxx(W_1,\prec)-\goxx(W_2,\prec)}\le2\cn{W_1-W_2}$.
\end{lemma}

\begin{proof}
  We may assume that $\sss$ is atomless. (Otherwise we consider  
$\sss\times\oi$.) In this case, writing $U_x$ for $\set{y:y\succ  x}$,
we have the alternative formula
\begin{equation}\label{er}
  \begin{split}
\goxx(W,\prec)
&=\iint W(x,z)\mu(U_x) \dd\mu(x)\dd\mu(z)
\\&\hskip4em{}
-\iiint_{x\prec y} W(x,z)W(y,z)\dd\mu(x)\dd\mu(y)\dd\mu(z)
\\&
=\iint \mu(U_x)W(x,z) \dd\mu(x)\dd\mu(z)
\\&\hskip4em{}
-\frac12\iiint W(x,z)W(y,z)\dd\mu(x)\dd\mu(y)\dd\mu(z)
\\&
=\iint \mu(U_x)W(x,z) \dd\mu(x)\dd\mu(z)
-\frac12\int \mW(z)^2\dd\mu(z) .
  \end{split}
\end{equation}
By the definition \eqref{cutnorm} of the cut norm,
\begin{equation*}
%  \begin{split}
\lrabs{\iint \mu(U_x)\bigpar{W_1(x,z)-W_2(x,z)} \dd\mu(x)\dd\mu(z)	}
\le\cn{W_1-W_2}.
%  \end{split}
\end{equation*}
Recalling that $|W_j|\le1$ and using \refL{L0} on $W_1-W_2$, 
\begin{multline*}
\lrabs{\ints \bigpar{\wis(z)^2- \wiis(z)^2}\dd\mu(z)}
\\
=\lrabs{\ints\bigpar{\wis(z)- \wiis(z)}\bigpar{\wis(z)+ \wiis(z)}\dd\mu(z)}
\\
\le 2\norm{\wis(z)- \wiis(z)}\qlis
\le2\cn{W_1-W_2}.  
\end{multline*}
Applying \eqref{er} to $W_1$ and $W_2$, the result follows.
\end{proof}

\begin{theorem}\label{TTx}
  Let $(\sss,\prec)$ be an ordered probability space and $W$ a kernel on
  $(\sss,\prec)$. Then $\goxx(W,\prec)=0$ if and only if $W$ is \aex{} equal to
 a \oivalued{} monotone kernel.
\end{theorem}

\begin{proof}
As usual, we may assume for simplicity that $\sss$ is atomless.
Suppose first $\goxx(W,\prec)=0$. 
For $a>0$,  let $E_a\=\set{(x,y)\in\sssq: a\le W(x,y)\le 1-a}$,
%the set where $W\in[a,1-a]$.
and, for $z\in\sss$, let $E_a(z)\=\set{x\in\sss:(x,z)\in E_a}$ be the
corresponding section.

If $x,y\in E_a(z)$, then $W(x,z)(1-W(y,z))\ge a^2$, and thus, 
for each $z$,
\begin{multline*}
\iint_{x\prec y} W(x,z)\bigpar{1-W(y,z)}\dd\mu(x)\dd\mu(y)
\\
\ge a^2\muu \bigset{(x,y)\in E_a(z)^2:x\prec y}
=\frac12 a^{2} \mu(E_a(z))^2.
\end{multline*}
Hence,
\begin{equation*}
  0=\goxx(W,\prec)\ge\ints \frac12 a^2\mu(E_a(z))^2\dd\mu(z),
\end{equation*}
and thus $\mu(E_a(z))=0$ for \aex{} $z$, so
$\mu\times\mu(E_a)=\ints\mu(E_a(z))\dd\mu(z)=0$. 
Consequently, $E_a$ is a null set for every $a>0$.
Hence, $W(x,y)\in\setoi$ a.e.
Thus $W$ is \aex{}  \oivalued, which implies that
$\gox(W,\prec)=\goxx(W,\prec)=0$; hence \refT{TMx} shows that $W$ is \aex{}
equal to a monotone kernel $W'$. Finally, $W'$ is \aex{} \oivalued, and thus
\aex{} equal to the \oivalued{} monotone kernel $\ett{W'> 0}$.

The converse is obvious.
\end{proof}

We also have an analogue of \refL{LA3}.
To prove this, we shall need the following `rearrangement' inequality.

\begin{lemma}\label{LTA3rearr}
Let $\prec$ and $<$ be two orders on an atomless probability space $\sss$,
and let $f$ 
be a bounded function on $\sss$.
If $x<y\implies f(x)\le f(y)$, then 
$\iint_{x\prec y}f(x)\dd\mu(x)\dd\mu(y) \ge 
\iint_{x< y}f(x)\dd\mu(x)\dd\mu(y)$.
\end{lemma}

\begin{proof}
Consider first one arbitrary order $\prec$.
Let $D_y\=\set{x:x\prec y}$ and set $\gf(y)\=\mu(D_y)$,
and let $D(t)$ be as in \refL{Lo}.
Then $D_y$ and $D(\gf(y))$ are two downsets with the same
measure, and thus they differ only by a null set, \cf{} \refL{Lo0}. 

Let $F(y)\=\int_{x\prec y} f(x)\dd\mu(x)$ and define $\ga(t)\=\int_{D(t)}f(x)\dd\mu(x)$. 
Then
\begin{equation*}
  F(y)=\int_{D_y} f = \int_{D(\gf(y))} f = \ga(\gf(y)).
\end{equation*}
It was noted in the proof of \refL{Lo} that if $X$ has distribution $\mu$,
then $\gf(X)$ has distribution $U(0,1)$. Equivalently,
the function $\gf:\sss\to\oi$ maps $\mu$ to the uniform measure on \oi. 
%(In other words, $\gf(y)$, regarded as a random variable, has a uniform
%distribution on \oi.) 
Hence,
\begin{equation*}
\iint_{x\prec y}f(x)\dd\mu(x)\dd\mu(y) 
=\ints F(y)\dd\mu(y)  
=\ints \ga(\gf(y))\dd\mu(y)  
=\intoi \ga(t)\dd t.
\end{equation*}
Now write $\ga=\ga_\prec$ and compare $\ga_\prec(t)$ and $\ga_<(t)$. Both
are integrals of $f$ over sets of measure $t$, and for $\ga_<$ the set is
such that if $x$ is in the set and $y$ is not, then $x<y$ and thus $f(x)\le
f(y)$.
It follows easily
that $\ga_<(t)$ is the minimum of $\int_E f\dd\mu$ over all set
$E$ of measure $t$, and thus in particular $\ga_<(t)\le\ga_\prec(t)$ for any
other order $\prec$. Consequently,
$\intoi\ga_<(t)\dd t \le \intoi\ga_\prec(t)\dd t$, and the result follows.
\end{proof}

\begin{lemma}
  \label{LTA3}
Let $<$  be a (measurable) order on $\sss$ and assume that
$x<y\implies \mW(x)\le\mW(y)$. 
Then, 
$\goxx(W,<)=\goxx(W)$.
\end{lemma}

\begin{proof}
We may again assume for simplicity that $\sss$ is atomless.
Let $\prec$ be any order on $\sss$. 
We again use \eqref{er}, which we write as
\begin{equation*}%\label{sjw}
  \goxx(W,\prec)=\ints\mu(U_x)\mW(x)\dd\mu(x)
-\frac12 \ints\mW(x)^2\dd\mu(x).
\end{equation*}
The second integral does not depend on $\prec$. Moreover, 
the first integral equals $\iint_{x\prec y}\mW(x)$, which by
\refL{LTA3rearr}
is minimized by taking $\prec$ equal to $<$.
Hence $\goxx(W,\prec)\ge\goxx(W,<)$, and the result follows.
\end{proof}

\begin{remark}
It follows by \eqref{goxxg} that the corresponding result holds for graphs
and $\goxb$: \ie, ordering the vertices by
their degrees achieves the minimum $\min_\prec\goxb(G,\prec)$.
\end{remark}

Our next result shows that $\goxx$
characterizes kernels that yield threshold graph limits.
Note the parallel and contrast to Theorems \refand{TM2}{TM2x}.

\begin{theorem}
  \label{TT2}
Let $W$
be a kernel on a \ps{} $\sss$ with at least one measurable order.
Then the following are equivalent. 
\begin{romenumerate}
\item \label{tt2-goxx}
$\goxx(W)=0$. 
\item \label{tt2-mono}
There exists an order $\prec$ on $\sss$ such that\/
$W$ is \aex{} equal to a \oivalued{} monotone kernel on $\sssp$.
\item \label{tt2-equiv}
$W$ is equivalent to a \oivalued{} monotone kernel on some \ops.
\item \label{tt2-oi}
$W$ is equivalent to a \oivalued{} monotone kernel on  $\oi$.
\item \label{tt2-gG}
$\gG_W$ is a threshold graph limit.
\end{romenumerate}
\end{theorem}

\begin{proof}
\ref{tt2-goxx}$\implies$\ref{tt2-mono}.
There exists a measurable order $\prec_0$ on $\sss$. 
As in the proof of \refC{CA3}, we define an order 
$\prec$ on $\sss$ by \eqref{ca3}. \refL{LTA3} applies and yields
$\goxx(W,\prec)=\goxx(W)=0$, 
and the result follows by \refT{TTx}.

\ref{tt2-mono}$\implies$\ref{tt2-goxx}.
\refT{TTx} yields $\goxx(W,\prec)=0$ and thus $\goxx(W)\le\goxx(W,\prec)=0$.

\ref{tt2-mono}$\iff$\ref{tt2-equiv}$\iff$\ref{tt2-oi}.
Every kernel equivalent to an \aex{} \oivalued{} kernel is itself \aex{}
\oivalued, see \refR{RLSz} and \cite{SJ249}. Furthermore,
arguing as in the proof of \refT{TTx},
a monotone kernel $W$ that
is \aex{} \oivalued{} is \aex{} equal to the
\oivalued{} monotone kernel $\ett{W>0}$.
Hence, 
\ref{tt2-mono}$\iff$\ref{tt2-equiv}$\iff$\ref{tt2-oi} follows from the
corresponding equivalences in \refT{TM2}.

\ref{tt2-oi}$\iff$\ref{tt2-gG}.
As noted in the introduction, this was proved by \citet{SJ238}.
\end{proof}

We need some more preparation before the proof of \refT{TT}.

\begin{lemma}
  \label{LT2}
Let $W_1$ and $W_2$ be kernels on a probability space $\sss$ with $W_1$
\oivalued, and let $W_1'$ be a \oivalued{} step kernel with $n$ steps. Then
\begin{equation*}
  \norm{W_1-W_2}\qliss
\le n^2\cn{W_1-W_2} +2\norm{W_1-W_1'}\qliss.
\end{equation*}
\end{lemma}
\begin{proof}
  Let $\set{A_i}_1^n$ be a partition of $\sss$ such that $W_1'$ is constant
0 or 1  on each $A_i\times A_j$.

If $W_1'=0$ on $A_i\times A_j$, then
\begin{multline*}
  	\iint_{A_i\times A_j}|W_1'-W_2|
=	
  \iint_{A_i\times A_j}W_2
\le
 \cn{W_1-W_2}+ \iint_{A_i\times A_j}W_1
\\
=
  \cn{W_1-W_2}+\iint_{A_i\times A_j}|W_1-W_1'|.
\end{multline*}

If $W_1'=1$ on $A_i\times A_j$, then
{\multlinegap=0pt
\begin{multline*}
	\iint_{A_i\times A_j}|W_1'-W_2|
=	
  \iint_{A_i\times A_j}(1-W_2)
\le
 \cn{W_1-W_2}+ \iint_{A_i\times A_j}(1-W_1)
\\
=
 \cn{W_1-W_2}+ \iint_{A_i\times A_j}|W_1-W_1'|.
  \end{multline*}
 }
 
Thus, in both cases
$	\iint_{A_i\times A_j}|W_1'-W_2|
\le  \iint_{A_i\times A_j}|W_1-W_1'|+\cn{W_1-W_2}$,
and summing over all $i$ and $j$ yields
\begin{equation*}
\normll{W_1'-W_2}
\le  \normll{W_1-W_1'}+n^2\cn{W_1-W_2}.
\end{equation*}
The result follows by
$\normll{W_1-W_2}\le \normll{W_1-W_1'}+\normll{W_1'-W_2}$.
\end{proof}

\begin{lemma}
  \label{LT3}
Let $W$ and $W_1,W_2,\dots$ be kernels on a probability space $\sss$, and
assume that $W$ is \oivalued. Then $\cn{W_n-W}\to0$ as \ntoo{} if and only
if
$\normlss{W_n-W}\to0$.
\end{lemma}

\begin{proof}
  Assume $\cn{W_n-W}\to0$.
$W$ is the indicator function $\etta_A$ of a measurable set $A\subseteq\sssq$. Any
  such set can be approximated in measure by a finite disjoint union of
  rectangle sets $\bigcup_i A_i\times B_i$, and we may assume that this set
  is symmetric since $A$ is; in other words, given any
  $\eps>0$, there exists a \oivalued{} step kernel $W'$ such that
  $\normll{W-W'}<\eps$. 
Let the corresponding partition have $N=N(\eps)$ parts.
\refL{LT2} then yields
\begin{equation*}
  \normll{W-W_n}\le N^2\cn{W-W_n}+2\eps\to2\eps
\end{equation*}
as \ntoo. Hence, $\limsup_\ntoo\normll{W-W_n}=0$.

The converse is obvious.
\end{proof}

\begin{proof}[Proof of \refT{TT}]

Note first that \ref{tt1} is equivalent to $\goxb(\gn)\to0$ by
\eqref{gox31}, and that $\goxb(\gn)=\goxx(W_{\gn})$ by \eqref{goxxg}. 

\ref{tt1}$\implies$\ref{ttt}.
Assume \ref{tt1} and consider a subsequence that converges.
We thus assume that there exists a graph limit $\gG$ with
$\gn\to\gG$. 
Let $W$ be a kernel on \oi{} representing $\gG$.

We have $\gn\to W$, and thus $\dcut(\wgn,W)\to0$.
Moreover, by \cite[Lemma 5.3]{BCLSV1} we may choose the labelling of the
vertices in $\gn$ such that 
\begin{equation}\label{ckw}
\cn{\wgn-W}\to0.
\end{equation}
This labelling yields an order $<$ on $V(\gn)$. 
Let $\prec$ be an order on $V(\gn)$ achieving the minimum in \eqref{gox} for
$\goxb(\gn)$, \ie, such that 
\begin{equation}\label{ckp}
\goxb(\gn,\prec)=\goxb(\gn)=o(1).  
\end{equation}

In general $\prec$ differs from $<$, but it clearly corresponds to some
order $\prec_\nu$ on $\oi$ and, by \eqref{goxxg} again,
\begin{equation}\label{ckoi}
 \goxb(\gn,\prec)=\gox(\wgn,\prec_\nu)=\goxx(\wgn,\prec_\nu). 
\end{equation}

By \refL{LT1} and \eqref{ckw}--\eqref{ckoi}, we then have
\begin{equation*}
  \goxx(W,\prec_\nu)\le \goxx(\wgn,\prec_\nu)+2\cn{W-\wgn}
%= \goxb(\gn,\prec)+2\cn{W-\wgn}
\to0,
\end{equation*}
as \nutoo; hence $\goxx(W)=0$ and $\gG=\gG_W$ is a threshold graph limit by
\refT{TT2}.

\ref{ttt}$\implies$\ref{ttd}
Suppose that \ref{ttd} fails; then there exists $\eps>0$ and a subsequence 
for which $\de(\gn,\cT)>\eps|\gn|^2$. We may select a subsubsequence such that
$\gn$ converges; we shall show that \ref{ttt} implies \ref{ttd} in this
case, which yields a contradiction.

Suppose then that $\gn\to\gG$
for some graph limit $\gG$, and that \ref{ttt} holds.  By assumption, $\gG$ is a threshold graph
limit.
Let $W$ be a kernel on $\oi$ representing $\gG$. By the result of
\citet{SJ238} discussed in the introduction, we may choose $W$ to be
monotone and \oivalued.

We have $\gn\to W$, and thus $\dcut(\wgn,W)\to0$.
As above, by \cite[Lemma 5.3]{BCLSV1} we may choose the labelling of the
vertices in $\gn$ such that $\cn{\wgn-W}\to0$.
By \refL{LT3}, this implies $\normll{\wgn-W}\to0$.

Since, by assumption, $\gG$ is a threshold graph limit, there exists a
sequence of threshold graphs $\gn'$ such that $\gn'\to\gG$, and we may
further assume that $|\gn'|=|\gn|$. (For example, we may \as{} take $\gn'$
as the random graph $G(n_\nu,W)$ with $n_\nu=|\gn|$.)
Then also  $\dcut(\wgni,W)\to0$, and by
\cite[Lemma 5.3]{BCLSV1} again we may choose the labelling of the
vertices in $\gn'$ such that $\cn{\wgni-W}\to0$, and thus by \refL{LT3}
$\normll{\wgni-W}\to0$.
Consequently,
\begin{equation*}
  \normll{\wgn-\wgni}
\le   \normll{\wgn-W} +   \normll{W-\wgni}\to0.
\end{equation*}

We may identify the vertex sets of $\gn$ and $\gn'$.
Then
\begin{equation*}
  \de(\gn,\cT)\le 
\bigabs{E(\gn)\setdiff E(\gn')}
=\tfrac12|\gn|^2\normll{\wgn-\wgni}
=o(|\gn|^2).
\end{equation*}

\ref{ttd}$\iff$\ref{tte} by the definition \eqref{dea}.  

\ref{tte}$\iff$\ref{ttl1} by
\begin{equation*}
 \norm{W_{\gn}-W_{\gn'}}\qliss
=2\abs{\gn}^{-2} \bigabs{E(\gn)\setdiff E(\gn')}.
\end{equation*}

\ref{ttl1}$\implies$\ref{ttcn} since $\cn{\cdot}\le\norm{\cdot}\qliss$.

\ref{ttcn}$\implies$\ref{tt1}.
Let $<$ be the order on $V(\gn)=V(\gn')$ defined by the degrees of the
vertices in $\gn'$.
Then, since $\gn'$ is a threshold graph, 
$N_{\gn'}(v)\subseteq N_{\gn'}(w)\cup\set w$
  whenever $v<w$, and thus $\goxa(\gn',<)=0$ by \eqref{gox30}.
%Hence, $\goxb(\gn',<)=O(1/n_\nu)=o(1)$ by \eqref{gox31}.

By \eqref{gox31}, \eqref{goxxg} and \refL{LT1},
%\begin{equation*}
%  \goxb(\gn,<)=\goxx(\wgn,<)
%\le \goxx(\wgni,<)+2\cn{\wgn-\wgni} 
%=0+o(1).
%\end{equation*}
\begin{equation*}
  \begin{split}
  \goxa(\gn,<)
&
=   \goxa(\gn,<)-  \goxa(\gn',<)
%\\&
=   \goxb(\gn,<)-  \goxb(\gn',<)+o(1)
\\&
=\goxx(\wgn,<)	-\goxx(\wgni,<)	+o(1)
\\&
\le 2\cn{\wgn-\wgni} +o(1)
=o(1).
  \end{split}
\end{equation*}
Hence  $\goxa(\gn)\to0$.
\end{proof}

\newcommand\AAP{\emph{Adv. Appl. Probab.} }
\newcommand\JAP{\emph{J. Appl. Probab.} }
\newcommand\JAMS{\emph{J. \AMS} }
\newcommand\MAMS{\emph{Memoirs \AMS} }
\newcommand\PAMS{\emph{Proc. \AMS} }
\newcommand\TAMS{\emph{Trans. \AMS} }
\newcommand\AnnMS{\emph{Ann. Math. Statist.} }
\newcommand\AnnPr{\emph{Ann. Probab.} }
\newcommand\CPC{\emph{Combin. Probab. Comput.} }
\newcommand\JMAA{\emph{J. Math. Anal. Appl.} }
\newcommand\RSA{\emph{Random Struct. Alg.} }
\newcommand\ZW{\emph{Z. Wahrsch. Verw. Gebiete} }
\newcommand\DMTCS{\jour{Discr. Math. Theor. Comput. Sci.} }

\newcommand\AMS{Amer. Math. Soc.}
\newcommand\Springer{Springer-Verlag}
\newcommand\Wiley{Wiley}

\newcommand\vol{\textbf}
\newcommand\jour{\emph}
\newcommand\book{\emph}
\newcommand\inbook{\emph}
\def\no#1#2,{\unskip#2, no. #1,} %(typeset after year) 
\newcommand\toappear{\unskip, to appear}

\newcommand\webcite[1]{%\hfil  %???
   %\penalty0 %???
\texttt{\def~{{\tiny$\sim$}}#1}\hfill\hfill}
\newcommand\webcitesvante{\webcite{http://www.math.uu.se/~svante/papers/}}
\newcommand\arxiv[1]{\webcite{arXiv:#1.}}

\def\nobibitem#1\par{}

\end{document}